\documentclass[a4paper]{amsart}
\usepackage{wrapfig}
\usepackage[dvips]{graphicx}
\usepackage{amsmath,amsthm,amssymb,amscd}
\usepackage{mathrsfs}
\usepackage[all]{xy}
\theoremstyle{definition}
\newtheorem{thm}{Theorem}[section]
\newtheorem{Def}[thm]{Definition}
\newtheorem{pro}[thm]{Proposition}
\newtheorem{cor}[thm]{Corollary}
\newtheorem{lem}[thm]{Lemma}

\newtheorem{rem}[thm]{Remark}

\newtheorem*{mainthm}{Theorem A}
\newtheorem*{mainthm2}{Theorem B}
\theoremstyle{definition}

\begin{document}
\title[$\mathcal{W}$-absorbing actions of 
finite groups]{$\mathcal{W}$-absorbing actions of 
finite groups on the Razak-Jacelon algebra}
\author{Norio Nawata}
\address{Department of Pure and Applied Mathematics, Graduate School of Information Science 
and Technology, The University of Osaka, Yamadaoka 1-5, Suita, Osaka 565-0871, Japan}
\email{nawata@ist.osaka-u.ac.jp}
\thanks{This work was supported by JSPS KAKENHI Grant Number JP20K03630}

\begin{abstract}
We say that a countable discrete group action $\alpha$ on a C$^*$-algebra $A$ 
is \textit{$\mathcal{W}$-absorbing} if there exist a C$^*$-algebra $B$ 
and an action $\beta$ on $B$ such that $\alpha$ is cocycle conjugate to  
$\beta\otimes \mathrm{id}_{\mathcal{W}}$ on $B\otimes \mathcal{W}$ where $\mathcal{W}$ is 
the Razak-Jacelon algebra. In this paper, we completely classify outer 
$\mathcal{W}$-absorbing actions of finite groups  on $\mathcal{W}$ up to conjugacy and 
cocycle conjugacy. 
\end{abstract}
\maketitle

\section*{Introduction}

Jones completely classified actions  of finite groups on the injective II$_1$ factor  up to conjugacy in 
\cite{Jones} (see also \cite{C3}). In this paper, we study a C$^*$-analog of this result. 

The Razak-Jacelon algebra $\mathcal{W}$ is the simple separable nuclear monotracial C$^*$-algebra 
that is $KK$-equivalent to $\{0\}$ (see \cite{EGLN}, \cite{J} and \cite{Na4}). 
We can regard $\mathcal{W}$ as a C$^*$-analog of the injective II$_1$ factor and a monotracial 
analog of the Cuntz algebra $\mathcal{O}_2$. 
There exist $K$-theoretical difficulties for studying group actions on C$^*$-algebras (see \cite{I} 
for details). Indeed, Barlak and Szab\'o showed that the UCT problem for separable nuclear 
C$^*$-algebras can be reduced to studying certain actions of finite cyclic groups on $\mathcal{W}$ 
in \cite{BS}. On the other hand, Gabe and Szab\'o classified outer actions of 
discrete amenable groups on Kirchberg algebras up to cocycle conjugacy by using equivariant 
$KK$-theory in \cite{GS}. 
It is an interesting open problem to generalize this classification to actions on ``classifiable'' stably 
finite C$^*$-algebras. As a first step, we shall study ``$\mathcal{W}$-absorbing'' actions on 
$\mathcal{W}$. 
We say that a countable discrete group action $\alpha$ on a C$^*$-algebra $A$ 
is \textit{$\mathcal{W}$-absorbing} if there exist a C$^*$-algebra $B$ 
and an action $\beta$ on $B$ such that $\alpha$ is cocycle conjugate to  
$\beta\otimes \mathrm{id}_{\mathcal{W}}$ on $B\otimes \mathcal{W}$. 
Since $\mathcal{W}$-absorbing actions have no $K$-theoretical obstructions, we can clarify 
difficulties due to stable finiteness by studying $\mathcal{W}$-absorbing actions. Note that 
there exist actions on $\mathcal{W}$ that are not $\mathcal{W}$-absorbing 
(as one can expect from the result above by Barlak and Szab\'o). Indeed, there exist uncountably many 
non-cocycle conjugate outer actions of $\mathbb{Z}_2$ on $\mathcal{W}$ that are not 
$\mathcal{W}$-absorbing by \cite[Example 5.6 and Remark 5.7]{Na0}. 

In this paper, we completely classify outer $\mathcal{W}$-absorbing actions of finite groups 
on $\mathcal{W}$ up to conjugacy and cocycle conjugacy. 
This classification can be regarded as a C$^*$-analog of Jones' classification. 
Every action $\alpha$ of a discrete group $\Gamma$ on $\mathcal{W}$ induces an action 
$\tilde{\alpha}$ of $\Gamma$ on the injective II$_1$ factor $\pi_{\tau_{\mathcal{W}}}(\mathcal{W})^{''}$. 
It is clear that if actions $\alpha$ and $\beta$ of $\Gamma$ on $\mathcal{W}$ are 
conjugate (resp. cocycle conjugate), then $\tilde{\alpha}$ and $\tilde{\beta}$ are conjugate 
(resp. cocycle conjugate). 
The invariant of Jones' classification for an action $\delta$ of a finite group $\Gamma$ on the injective II$_1$ 
factor is a triplet consisting of the normal subgroup 
$N(\delta):=\{g\in \Gamma\; |\; \delta_g\; \text{is an inner automorphism} \}$ of $\Gamma$, 
the characteristic invariant $\Lambda(\delta)$ and the inner invariant $i(\delta)$. 
In this paper, we show the following classification theorem. 

\begin{mainthm}
(Theorem \ref{thm:main-1} and Theorem \ref{thm:main-2}.) \ \\
Let $\alpha$ and $\beta$ be outer $\mathcal{W}$-absorbing actions of a finite group 
$\Gamma$ on 
$\mathcal{W}$. Then \ \\
(1) $\alpha$ is cocycle conjugate to $\beta$ if and only if 
$(N(\tilde{\alpha}), \Lambda (\tilde{\alpha}))
=(N(\tilde{\beta}), \Lambda (\tilde{\beta}))$, \ \\
(2) $\alpha$ is conjugate to $\beta$ if and only if 
$(N(\tilde{\alpha}), \Lambda (\tilde{\alpha}), i(\tilde{\alpha}))
=(N(\tilde{\beta}), \Lambda (\tilde{\beta}), i(\tilde{\beta}))$. 
\end{mainthm}

The following rigidity type theorem is an immediate consequence of the theorem above. 

\begin{mainthm2}
(Corollary \ref{cor:main-1} and Corollary \ref{cor:main-2}.) \ \\
Let $\alpha$ and $\beta$ be outer $\mathcal{W}$-absorbing actions of a finite group 
$\Gamma$ on 
$\mathcal{W}$. Then \ \\
(1) $\alpha$ is cocycle conjugate to $\beta$ if and only if 
$\tilde{\alpha}$ is cocycle conjugate to $\tilde{\beta}$, \ \\
(2) $\alpha$ is conjugate to $\beta$ if and only if 
$\tilde{\alpha}$ is conjugate to $\tilde{\beta}$. 
\end{mainthm2}

A key ingredient in the proof of the classification theorem above is a kind of generalization 
(Theorem \ref{thm;absorption} or Corollary \ref{cor;absorption}) of equivariant Kirchberg-Phillips type 
absorption results for $\mathcal{W}$ in \cite{Na3} and \cite{Na5}. 
This generalization can be regarded as a monotracial analog (for the case of finite groups) 
of equivariant Kirchberg-Phillips type absorption  
theorems for $\mathcal{O}_2$ in \cite{I1}, \cite{Sza4} and \cite{Su}. (See also \cite{PS}.) 

This paper is organized as follows. In Section \ref{sec:pre}, we collect notations and some results. 
Furthermore, we study coboundaries of outer actions of finite groups on 
simple separable exact monotracial stably projectionless  $\mathcal{Z}$-stable  C$^*$-algebras. 
In Section \ref{sec:mod}, we construct outer actions with arbitrary invariants in the classification for 
outer $\mathcal{W}$-absorbing actions up to cocycle conjugacy and conjugacy. 
Note that we need model actions in the classification up to cocycle conjugacy for the proof of 
Theorem A. In Section \ref{sec:app}, we show that certain actions of finite groups on $\mathcal{W}$ 
with trivial characteristic invariant  (in $\pi_{\tau_{\mathcal{W}}}(\mathcal{W})^{''}$) have a kind of 
approximate representability. This result enables us to show an existence type theorem in 
the next section. In Section \ref{sec:abs}, we show the absorption result explained above by 
using Szab\'o's approximate cocycle intertwining argument in \cite{Sza7} (see also \cite{Ell2} and 
\cite[Section 5]{Na5}). 
In Section \ref{sec:cla}, we show the main classification theorem. Many arguments in the proof of 
the classification theorem are based on Jones' arguments in \cite{Jones}. 

\section{Preliminaries}\label{sec:pre}
\subsection{Finite group actions on monotracial C$^*$-algebras}

Let $T_1(A)$ denote the tracial state space of a C$^*$-algebra $A$. For any $\tau\in T_1(A)$, 
define $d_{\tau}(a):=\lim_{n\to\infty}\tau (a^{1/n})$ for any positive element $a$ in $A$. 
We say that a C$^*$-algebra $A$ is \textit{monotracial} if $T_1(A)=\{\tau_A\}$ and 
$A$ has no unbounded traces. (Let $\tau_A$ denote the unique tracial state on $A$ 
unless otherwise specified.) 
We denote by $(\pi_{\tau_A}, H_{\tau_{A}}, \xi_{\tau_A})$ the Gelfand-Naimark-Segal 
(GNS) representation of $\tau_A$. 
If $A$ is a simple infinite-dimensional monotracial C$^*$-algebra, then 
$\pi_{\tau_{A}}(A)^{''}$ is a factor of type II$_1$. For $x\in \pi_{\tau_{A}}(A)^{''}$, define 
$\| x \|_2:= \tilde{\tau}_{A}(x^*x)^{1/2}$ where $\tilde{\tau}_A$ is the unique normal extension of 
$\tau_A$ on $\pi_{\tau_{A}}(A)^{''}$. 
Let $A^{\sim}$ and $M(A)$ denote the unitization algebra of $A$ and the multiplier algebra of $A$, 
respectively. Note that $A=A^{\sim}=M(A)$ if $A$ is unital. 
Let $U(M(A))$ denote the unitary group of $M(A)$. 
For any $u\in U(M(A))$, define an automorphism $\mathrm{Ad}(u)$ 
of $A$ by $\mathrm{Ad}(u)(x)=uxu^*$ for any $x\in A$. 

Let $\alpha$ and $\beta$ be actions of a discrete group $\Gamma$ on 
simple infinite-dimensional monotracial C$^*$-algebras $A$ and $B$, respectively. 
One says that $\alpha$ is \textit{conjugate} to $\beta$ if there exists an isomorphism $\theta$ 
from $A$ onto $B$ such that $\theta \circ \alpha_g= \beta_g\circ \theta$ for any $g\in \Gamma$. 
A map $u$ from $\Gamma$ to $U(M(B))$ 
is said to be a \textit{$\beta$-cocycle} if $u_{gh}=u_g\beta_g (u_h)$ for any $g, h\in \Gamma$. 
We say that $\alpha$ is \textit{cocycle conjugate} to $\beta$ if there exist 
an isomorphism $\theta$ from $A$ onto $B$ and a $\beta$-cocycle $u$ such that 
$\mathrm{Ad}(u_g)\circ \beta_g \circ \theta= \theta\circ \alpha_g$ for any $g\in \Gamma$. 
Since $A$ and $B$ are monotracial, $\alpha$ and $\beta$ induce actions $\tilde{\alpha}$ and 
$\tilde{\beta}$ on $\pi_{\tau_{A}}(A)^{''}$ and $\pi_{\tau_{B}}(B)^{''}$, respectively. 
If $\alpha$ and $\beta$ are (resp. cocycle) conjugate, then $\tilde{\alpha}$ and $\tilde{\beta}$ 
are (resp. cocycle) conjugate. 
Put 
$$
N(\alpha):= \{g\in \Gamma\; |\; \alpha_g=\mathrm{Ad}(u) \text{ for some } u\in U(M(A))\}.
$$
Then $N(\alpha)$ is a normal subgroup of $\Gamma$. 
We say that $\alpha$ is \text{outer} if $N(\alpha)=\{\iota\}$ where $\iota$ is the identity element in 
$\Gamma$. 
If $\alpha$ and $\beta$ are cocycle conjugate, then $N(\alpha)=N(\beta)$ and  
$N(\tilde{\alpha})=N(\tilde{\beta})$. 
 
Assume that $\Gamma$ is a finite group. 
Let $A^{\alpha}$ and $A\rtimes_{\alpha}\Gamma$ denote the fixed point subalgebra and 
the (reduced) crossed product C$^*$-algebra, respectively. 
Let $E_{\alpha}$ be the canonical conditional expectation from $A\rtimes_{\alpha}\Gamma$
onto $A$ 
and 
$$
e_{\alpha}:= \frac{1}{|\Gamma|} \sum_{g\in \Gamma} \lambda_g^{\alpha}\in M(A\rtimes_{\alpha}\Gamma)
$$
where $\lambda_g^{\alpha}$ is the implementing unitary element of $\alpha_g$ in  
$M(A\rtimes_{\alpha}\Gamma)$. 
Then $e_{\alpha}(A\rtimes_{\alpha}\Gamma)e_{\alpha}$ is isomorphic to $A^{\alpha}$. 
It is easy to see that $\pi_{\tau_A\circ E_{\alpha}}(A\rtimes_{\alpha}\Gamma)^{''}$ is isomorphic to 
$\pi_{\tau_{A}}(A)^{''}\rtimes_{\tilde{\alpha}}\Gamma$. 
Indeed, there exists a unitary operator $U$ from $H_{\tau_A\circ E_{\alpha}}$ onto 
$H_{\tau_A}\otimes \ell^2(\Gamma)$ such that 
$$
U\xi_{x}= \sum_{g\in\Gamma}\xi_{\alpha_g^{-1}(a_g)}\otimes \delta_g
$$
for any $x=\sum_{g\in\Gamma}a_g\lambda^{\alpha}_g \in A\rtimes_{\alpha} \Gamma$ where 
$\{\delta_g\}_{g\in\Gamma}$ is the canonical orthonormal basis of $\ell^2(\Gamma)$. 
Then we have $U\pi_{\tau_A\circ E_{\alpha}}(A\rtimes_{\alpha}\Gamma)^{''}U^*
=\pi_{\tau_{A}}(A)^{''}\rtimes_{\tilde{\alpha}}\Gamma$.  
Note that we can regard $A\rtimes_{\alpha}\Gamma$ and $M(A\rtimes_{\alpha}\Gamma)$ 
as  subalgebras of $\pi_{\tau_A\circ E_{\alpha}}(A\rtimes_{\alpha}\Gamma)^{''}\cong 
\pi_{\tau_{A}}(A)^{''}\rtimes_{\tilde{\alpha}}\Gamma$ because we assume that $A$ is simple. 
(Note that the simplicity of $A$ implies that $\pi_{\tau_A\circ E_{\alpha}}$ is a faithful 
nondegenerate representation of $A\rtimes_{\alpha}\Gamma$.) 

Let $\rho$ be the right regular representation of $\Gamma$ on 
$\ell^2(\Gamma)$, and let 
$\mathbb{K}(\ell^2(\Gamma))$ be the C$^*$-algebra of compact operators on $\ell^2(\Gamma)$. 
Since $\Gamma$ is a finite group, we can identify $\mathbb{K}(\ell^2(\Gamma))$ with 
$M_{|\Gamma|}(\mathbb{C})$ in the standard way. 
Let $\{e_{g,h}^{\Gamma}\}_{g,h\in\Gamma}$ denote the canonical matrix units of 
$\mathbb{K}(\ell^2(\Gamma))=M_{|\Gamma|}(\mathbb{C})$. If $\alpha$ is cocycle conjugate to $\beta$, 
then $\alpha\otimes \mathrm{Ad}(\rho)$ is conjugate to 
$\beta\otimes \mathrm{Ad}(\rho)$. 
Although this fact can be regarded as a consequence of the Imai-Takai duality theorem 
in \cite{IT}, we can show this fact by direct computations. Indeed, let $\theta$ be an isomorphism 
from $A$ onto $B$ and $u$ a $\beta$-cocycle such that $\mathrm{Ad}(u_g)\circ \beta_g \circ \theta
=\theta \circ \alpha_g$ for any $g\in \Gamma$. Define a homomorphism $\Theta$ from 
$A\otimes\mathbb{K}(\ell^2(\Gamma))$ to $B\otimes\mathbb{K}(\ell^2(\Gamma))$ by
$$
\Theta \left( \sum_{h,k\in \Gamma} a_{h,k}\otimes e_{h,k}^{\Gamma}\right)
= \sum_{h,k\in\Gamma} u_{h^{-1}}^*\theta (a_{h,k})u_{k^{-1}}\otimes e_{h, k}^{\Gamma}
$$
for any $ \sum_{h,k\in \Gamma} a_{h,k}\otimes e_{h,k}^{\Gamma}\in 
A\otimes\mathbb{K}(\ell^2(\Gamma))$. 
Then $\Theta$ is an isomorphism from $A\otimes\mathbb{K}(\ell^2(\Gamma))$ onto 
$B\otimes\mathbb{K}(\ell^2(\Gamma))$ such that 
$\Theta \circ (\alpha_g\otimes \mathrm{Ad}(\rho_g))
=(\beta_g\otimes \mathrm{Ad}(\rho_g))\circ \Theta$ for any $g\in\Gamma$. 
Also, $(A\otimes\mathbb{K}(\ell^2(\Gamma)))^{\alpha\otimes \mathrm{Ad}(\rho)}$ 
is  isomorphic to $A\rtimes_{\alpha}\Gamma$. 

\subsection{Characteristic invariant and inner invariant} 
We shall recall some results in \cite{Jones}. 
We refer the reader to \cite{Jones}, \cite{Jones2} and \cite{Oc} for details and general cases. 
For a finite set $P$, let $M(P)$ be the set of probability measures on $P$. 
Let $\mathbb{T}$ denote the unit circle in the complex plane. 

For a finite group $\Gamma$ and a normal subgroup $N$ of $\Gamma$, let $Z(\Gamma, N)$ be 
the set of  pairs $(\lambda, \mu)$ where $\lambda$ is a map from $\Gamma\times N$ to 
$\mathbb{T}$ and $\mu$ is a map from $N\times N$ to $\mathbb{T}$ satisfying the following 
relations: 
\begin{align*}
& (\mathrm{i})\; \mu (h_1, h_2)\mu (h_1h_2, h_3)=\mu (h_2,h_3)\mu(h_1, h_2h_3) \\
& (\mathrm{ii})\; \lambda(g_1, h_1h_2)\overline{\lambda(g_1, h_1)}\overline{\lambda(g_1, h_2)}=
\mu (h_1, h_2)\overline{\mu(g_1^{-1}h_1g_1, g_1^{-1}h_2g_1)} \\
& (\mathrm{iii})\; \lambda(g_1g_2, h_1)=\lambda(g_1, h_1)\lambda(g_2, g_1^{-1}h_1g_1) \\
& (\mathrm{iv})\; \lambda (h_1, h_2)= \mu(h_1, h_1^{-1}h_2h_1)\overline{\mu (h_2, h_1)} \\
& (\mathrm{v})\; \lambda (\iota, h_1)= \lambda (g_1, \iota)= \mu (\iota, h_1)= \mu (h_1, \iota)=1
\end{align*}
for any $g_1, g_2\in \Gamma$ and $h_1, h_2\in N$. 
With pointwise operations, $Z(\Gamma, N)$ is an abelian group. 
Note that $\mu$ is nothing but a 2-cocycle of $N$. 
For a map $\eta$ from $N$ to $\mathbb{T}$ satisfying $\eta (\iota)=1$, set 
$$
\lambda^{\eta} (g,h_1):= \eta(h_1)\overline{\eta(g^{-1}h_1g)}, \quad
\mu^{\eta}(h_1, h_2)=\eta(h_1h_2)\overline{\eta (h_2)\eta (h_1)} 
$$
for any $g\in\Gamma$, $h_1, h_2\in N$. Then 
$B(\Gamma, N):=\{(\lambda^{\eta}, \mu^{\eta})\; |\; \eta: N\to \mathbb{T},\; \eta (\iota)=1\}$ is 
a subgroup of $Z(\Gamma, N)$.  Let $\Lambda(\Gamma, N)$ be the quotient group 
$Z(\Gamma, N)/ B(\Gamma, N)$. 

Let $(\lambda, \mu)\in Z(\Gamma, N)$. 
Since $\mu$ is a 2-cocycle of $N$, we can consider the twisted group algebra $\mathbb{C}_{\mu} N$. 
Note that $\mathbb{C}_{\mu} N$ can be regarded as a von Neumann algebra because 
$N$ is a finite group. Define an action $\psi^{\lambda}$ of $\Gamma$ on $\mathbb{C}_{\mu} N$ by 
$$
\psi^{\lambda}_g\left( \sum_{h\in N}c_h h\right) := \sum_{h\in N}\lambda(g,h)c_{g^{-1}hg}h
$$ 
for any 
$g\in \Gamma$ and $\sum_{h\in N}c_h h\in\mathbb{C}_{\mu}N$.
By \cite[Proposition 1.3.1]{Jones}, $(\mathbb{C}_{\mu}N)^{\psi^{\lambda}}$ is a 
finite-dimensional abelian von Neumann algebra. 
Let $P_{\lambda, \mu}$ be the set of minimal projections in $(\mathbb{C}_{\mu}N)^{\psi^\lambda}$. 
Since $p\in P_{\lambda, \mu}$ is an element in $\mathbb{C}_{\mu}N$, $p$ can be written as a linear 
combination of elements in $N$. Let $X_p(h)$ denote the coefficients of $p\in P_{\lambda, \mu}$, that is, 
$p=\sum_{h\in N}X_{p}(h)h$. See \cite[Proposition 1.3.2]{Jones} for properties of $X_p(h)$. 
Let $H^1(N)^{\Gamma}$ be the group of homomorphisms $\eta$ from $N$ to $\mathbb{T}$ satisfying 
$\eta (ghg^{-1})=\eta (h)$ for any $g\in\Gamma$ and $h\in N$.  
Define an action $\partial$ of $H^{1}(N)^{\Gamma}$ on $\mathbb{C}_{\mu}N$ by 
$$
\partial_{\eta} \left( \sum_{h\in N}c_h h\right)=\sum_{h\in N}\eta (h)  c_hh
$$
for any $\eta\in H^1(N)^{\Gamma}$ and $\sum_{h\in N}c_h h\in\mathbb{C}_{\mu}N$. 
Since we have $\psi^{\lambda}_g\circ \partial_\eta
=\partial_\eta \circ \psi^{\lambda}_g$ for any $g\in \Gamma$ and $\eta\in H^1(N)^{\Gamma}$, 
$\partial$ induces an action on $P_{\lambda, \mu}$. 
Define an equivalence relation $\sim$ on $M(P_{\lambda, \mu})$ by $m\sim m^{\prime}$ if 
there exists an element $\eta$ in $H^1(N)^{\Gamma}$ such that $m^{\prime}
=m\circ \partial_\eta$, and let 
$I(P_{\lambda, \mu}):=M(P_{\lambda, \mu})/\sim$. 

Let $\delta$ be an action of $\Gamma$ on a II$_1$ factor $M$. 
Choose a map $V$ from $N(\delta)$ to $U(M)$ such that $\delta_h=\mathrm{Ad}(V_h)$ for any 
$h\in N(\delta)$.
For any $g\in \Gamma$ and $h\in N(\delta)$, there exists a complex number 
$\lambda_{\delta}(g, h)$ 
in $\mathbb{T}$ such that $\delta_g (V_{g^{-1}hg})=\lambda_{\delta} (g, h)V_{h}$ because $M$ is a factor. 
In a similar way, for any $h_1, h_2\in N(\delta)$, there exists a complex number $\mu_{\delta} (h_1, h_2)$ 
in $\mathbb{T}$ such that $V_{h_1}V_{h_2}=\mu_\delta (h_1, h_2)V_{h_1h_2}$. 
It is easy to see that $(\lambda_{\delta}, \mu_{\delta})$ is an element in $Z(\Gamma, N(\delta))$. 
Note that $(\lambda_{\delta}, \mu_{\delta})\in Z(\Gamma, N(\delta))$ depends on the choice of $V$. 
Put $\Lambda (\delta):=[(\lambda_{\delta}, \mu_{\delta})]\in \Lambda(\Gamma, N(\delta))$. 
Then $\Lambda (\delta)$ does not depend on the choice of $V$ and is a 
cocycle conjugacy invariant. (See \cite[Section 1.2]{Jones}.)
We call $\Lambda(\delta)$ 
the \textit{characteristic invariant} of $\delta$. 
It is easy to see that there exists a 
unitary representation $V$ of $N(\delta)$ on $M$  such that $\delta_{h}=\mathrm{Ad}(V_h)$ and 
$\delta_g(V_h)= V_{ghg^{-1}}$ for any $g\in \Gamma$ and $h\in N(\delta)$ if and only if 
$\Lambda(\delta)=[(\lambda_{\delta}, \mu_{\delta})]$ is trivial, that is, $(\lambda_{\delta}, \mu_{\delta})\in 
B(\Gamma, N(\delta))$. 

Define a homomorphism $\Phi_V$ from $\mathbb{C}_{\mu} N(\delta)$ to $M$ by 
$$
\Phi_V\left(\sum_{h\in N(\delta)}c_h h\right)=\sum_{h\in N(\delta)}c_h V_h
$$
for any $\sum_{h\in N(\delta)}c_h h\in\mathbb{C}_{\mu}N(\delta)$. 
By \cite[Corollary 2.1.2]{Jones}, the center $Z(M^{\delta})$ of $M^{\delta}$ is equal to 
$\Phi_V((\mathbb{C}_{\mu}N(\delta))^{\psi^{\lambda_{\delta}}})$. 
Define a probability measure $m_{V}(\delta)$ on $P_{\lambda_{\delta}, \mu_{\delta}}$ by 
$$
m_V(\delta)(p):= \tau_{M}(\Phi_V(p))
$$ 
for any $p\in P_{\lambda_{\delta}, \mu_{\delta}}$, and put  
$i(\delta):=[m_V(\delta)]\in I(P_{\lambda_{\delta}, \mu_{\delta}})$. Note that 
$i(\delta)$ does not depend on the choice of $V$ and is a conjugacy invariant. 
We call $i(\delta)$ the \textit{inner invariant} of $\delta$. 
Jones showed that $(N(\delta), \Lambda (\delta), i(\delta))$ is a complete conjugacy
invariant for actions of finite groups on the injective II$_1$ factor in \cite{Jones}. 

Define a map $\Pi_{V}$ from $\mathbb{C}_{\mu}N(\delta)$ to $M\rtimes_{\delta}\Gamma$ by
$$
\Pi_V\left(\sum_{h\in N(\delta)}c_h h\right)=\sum_{h\in N(\delta)}c_h V_h\lambda_h^{\delta*}.
$$ 
Then $\Pi_{V}$ induces an isomorphism from 
$(\mathbb{C}_{\mu}N(\delta))^{\psi^{\lambda_{\delta}}}$ onto $Z(M\rtimes_{\delta}\Gamma)$ by 
\cite[Corollary 2.2.2]{Jones}. 
Therefore the set of minimal projections in 
$Z(M\rtimes_{\delta}\Gamma)$ is equal to $\{ \Pi_V(p)\; |\; p\in P_{\lambda_{\delta}, \mu_{\delta}}\}$. 
In particular, $M\rtimes_{\delta}\Gamma$ is isomorphic to 
$\bigoplus_{p\in P_{\lambda_{\delta}, \mu_{\delta}}}\Pi_V(p)(M\rtimes_{\delta}\Gamma)$,
which is the direct sum of 
finitely many II$_1$ factors. (See \cite[Corollary 2.2.3]{Jones}.) 
For any $p\in P_{\lambda_{\delta}, \mu_{\delta}}$, 
define a function $\tau_{p}$ on $M\rtimes_{\delta}\Gamma$ 
by 
$$
\tau_{p}(T):=\frac{1}{X_p(\iota)}\tau_{M}(E_{\delta}(T\Pi_{V}(p)))
$$ 
for any $T\in M\rtimes_{\delta}\Gamma$.
Then $\tau_{p}$ is an extremal tracial state on $M\rtimes_{\delta}\Gamma$. 
Moreover, the set of extremal tracial states on $M\rtimes_{\delta}\Gamma$ is equal to 
$\{\tau_{p}\; |\; p\in P_{\lambda_\delta, \mu_{\delta}}\}$. 
Hence every tracial state on $M\rtimes_{\delta}\Gamma$ is normal because 
$P_{\lambda_{\delta}, \mu_{\delta}}$ is a finite set. 
Note that we have 
\begin{align*}
\tau_{p}\left(\sum_{g\in \Gamma}T_g \lambda_{g}^{\delta}\right)
&= \frac{1}{X_p(\iota)}\sum_{g\in \Gamma}\sum_{h\in N(\delta)}X_p(h)
\tau_{M}(T_g E_{\delta}(\lambda^{\delta}_gV_h\lambda_{h}^{\delta *})) \\
&=\frac{1}{X_p(\iota)}\sum_{g\in \Gamma}\sum_{h\in N(\delta)}X_p(h)
\tau_{M}(T_g\delta_{g}(V_h) E_{\delta}(\lambda_{gh^{-1}}^{\delta})) \\
&= \frac{1}{X_p(\iota)}\sum_{h\in N(\delta)}X_p(h)
\tau_{M}(T_hV_h) 
\end{align*}
for any $\sum_{g\in \Gamma}T_g \lambda_{g}^{\delta}\in M\rtimes_{\delta}\Gamma$ and 
$p\in P_{\lambda_\delta, \mu_{\delta}}$. 
In particular, we have 
$$
\tau_{p}(e_{\delta})
= \frac{1}{X_p(\iota)|\Gamma|}\tau_{M}\left( \Phi_V(p)\right)
=\frac{1}{X_p(\iota)|\Gamma|}m_V(\delta)(p) \eqno{(1.2.1)}
$$
for any $p\in P_{\lambda_\delta, \mu_{\delta}}$. 

\begin{rem}
Let $\alpha$ be an action of a finite group $\Gamma$ on a unital C$^*$-algebra $A$, and let 
$N$ be a normal subgroup of $\Gamma$. Assume that there exist an element 
$(\lambda, \mu)$ in $Z(\Gamma, N)$ and a map $v$ from $N$ to 
$U(A)$ such that $\alpha_h=\mathrm{Ad}(v_h)$, 
$\alpha_g (v_{g^{-1}hg})= \lambda(g,h)v_{h}$ and $v_{h}v_{k}=\mu(h, k)v_{hk}$ for any $g\in\Gamma$ and 
$h, k\in N$. Then we can define a homomorphism $\Phi_v$ from $(\mathbb{C}_{\mu}N)^{\psi^\lambda}$ 
to $Z(A^{\alpha})$ by 
$$
\Phi_v\left(\sum_{h\in N}c_h h\right)=\sum_{h\in N}c_h v_h
$$
for any $\sum_{h\in N}c_h h\in (\mathbb{C}_{\mu}N)^{\psi^\lambda}$ as above. 
\end{rem}

\subsection{Tracial state spaces and coboundaries}

The following proposition is a generalization of \cite[Proposition 1.5]{Na6}.

\begin{pro}\label{pro:tracial-state}
Let $A$ be a simple separable monotracial C$^*$-algebra, and let $\alpha$ be an outer action of 
a finite group $\Gamma$ on $A$. 
Then every tracial state on $A\rtimes_{\alpha}\Gamma$ is a restriction of a tracial state on 
$\pi_{\tau_{A}\circ E_{\alpha}}(A\rtimes_{\alpha}\Gamma)^{''}\cong 
\pi_{\tau_{A}}(A)^{''}\rtimes_{\tilde{\alpha}}\Gamma$. 
Furthermore, the set of extremal tracial states 
on $A\rtimes_{\alpha}\Gamma$ is equal to 
$\{\tau_{p}|_{A\rtimes_{\alpha}\Gamma} \; |\; p\in P_{\lambda_{\tilde{\alpha}}, \mu_{\tilde{\alpha}}}\}$. 
\end{pro}
\begin{proof}
Let $\tau$ be a tracial state on $A\rtimes_{\alpha}\Gamma$. 
Then $\pi_{\tau}$ is a faithful representation because $A\rtimes_{\alpha}\Gamma$ is 
simple by the outerness of $\alpha$. Since $A$ is monotracial and 
$A\subset A\rtimes_{\alpha}\Gamma$ is a nondegenerate inclusion, 
we have $\tau|_A=\tau_{A}$. Hence an inclusion $A\subset A\rtimes_{\alpha}\Gamma$ induces 
an injective homomorphism $\varphi$ from $\pi_{\tau_A}(A)^{''}$ to 
$\pi_{\tau}(A\rtimes_{\alpha}\Gamma)^{''}$. 
Define a homomorphism $\Phi$ from 
$\pi_{\tau_{A}\circ E_{\alpha}}(A\rtimes_{\alpha}\Gamma)^{''}
\cong\pi_{\tau_{A}}(A)^{''}\rtimes_{\tilde{\alpha}}\Gamma$ to 
$\pi_{\tau}(A\rtimes_{\alpha}\Gamma)^{''}$
by 
$$
\Phi \left( \sum_{g\in \Gamma}T_g \lambda_g^{\tilde{\alpha}}\right)= \sum_{g\in\Gamma}
\varphi (T_g) \pi_{\tau}(\lambda_g^{\alpha})
$$
for any $\sum_{g\in \Gamma}T_g \lambda_g^{\tilde{\alpha}}\in 
\pi_{\tau_{A}}(A)^{''}\rtimes_{\tilde{\alpha}}\Gamma$. Note that $\Phi|_{A\rtimes_{\alpha}\Gamma}$ 
can be regarded as the identity map because  $\pi_{\tau_A\circ E_{\alpha}}$ and $\pi_{\tau}$ 
are faithful representations and 
we have 
$
\Phi\left(\pi_{\tau_A\circ E_{\alpha}}(x)\right)
=\pi_{\tau}\left(x\right)
$
for any $x\in A\rtimes_{\alpha}\Gamma$. 
Since $\Phi$ is a unital homomorphism, 
$\tilde{\tau} \circ \Phi$ is a tracial state on 
$\pi_{\tau_{A}\circ E_{\alpha}}(A\rtimes_{\alpha}\Gamma)^{''}$. 
Therefore $\tau$ is a restriction of a tracial state on 
$\pi_{\tau_{A}\circ E_{\alpha}}(A\rtimes_{\alpha}\Gamma)^{''}$. 

Since every tracial state on $\pi_{\tau_{A}\circ E_{\alpha}}(A\rtimes_{\alpha}\Gamma)^{''}$ 
is normal (see the previous subsection), 
the restriction map $T_1(\pi_{\tau_{A}\circ E_{\alpha}}(A\rtimes_{\alpha}\Gamma)^{''})\ni 
\tau \mapsto \tau |_{A\rtimes_{\alpha}\Gamma}\in T_1(A\rtimes_{\alpha}\Gamma)$ is 
injective. 
Therefore the set of extremal tracial states 
on $A\rtimes_{\alpha}\Gamma$ is equal to 
$\{\tau_{p}|_{A\rtimes_{\alpha}\Gamma} \; |\; p\in P_{\lambda_{\tilde{\alpha}}, \mu_{\tilde{\alpha}}}\}$ 
by the results reviewed in the previous subsection. 
\end{proof}

Let  $A$ be a simple separable monotracial C$^*$-algebra, and let $\alpha$ be an outer action of 
a finite group $\Gamma$ on $A$. 
For an $\alpha$-cocycle $u$, put 
$$
e_{u}:= \frac{1}{|\Gamma |}\sum_{g\in\Gamma} u_g\lambda_{g}^{\alpha}\in M(A\rtimes_{\alpha}\Gamma).
$$
Then $e_{u}$ is a projection. 
An $\alpha$-cocycle $u$ is a \textit{coboundary} if there exists a unitary element $w$ in $M(A)$ 
such that $u_g=w\alpha_g(w^*)$ for any $g\in \Gamma$. Note that if 
an $\alpha$-cocycle $u$ is a coboundary, then $\mathrm{Ad}(u)\circ \alpha$ is conjugate to 
$\alpha$ because we have 
$
\mathrm{Ad}(u_g)\circ \alpha=  \mathrm{Ad}(w) \circ \alpha_g \circ \mathrm{Ad}(w)^{-1}
$
for any $g\in \Gamma$. 
We have the following proposition 
by the same argument as in the proof of \cite[Proposition 3.1.2]{Jones}

\begin{pro}\label{pro:Jones-3.1.2}
Let $u$ be an $\alpha$-cocycle. 
Then $u$ is a coboundary if and only if $e_{u}$ is Murray-von Neumann
equivalent to $e_{\alpha}$ in $M(A\rtimes_{\alpha}\Gamma)$. 
\end{pro}

The following proposition is analogous to \cite[Theorem 3.1.3]{Jones}. 

\begin{pro}\label{pro:coboundary}
Let $A$ be a simple separable exact monotracial stably projectionless  
$\mathcal{Z}$-stable  C$^*$-algebra, and let $\alpha$ be an outer action of a finite group 
$\Gamma$ on $A$. Assume that $V$ is a map from $N(\tilde{\alpha})$ to 
$U(\pi_{\tau_A}(A)^{''})$ such that $\tilde{\alpha}_h=\mathrm{Ad}(V_h)$ for any 
$h\in N(\tilde{\alpha})$. Then an $\alpha$-cocycle $u$ is a coboundary if and only if 
$$
\tilde{\tau}_A(\pi_{\tau_A}(u_h)V_h)= \tilde{\tau}_A(V_h)
$$
for any $h\in N(\tilde{\alpha})$.
\end{pro}
\begin{proof}
Easy computations show the only if part. We shall show the if part. 
By the assumption, $A\rtimes_{\alpha}\Gamma$ is simple, separable and exact. 
\cite[Theorem 5.20]{Sza6}(see also \cite[Theorem 1.1]{Sa2}) implies that 
$A\rtimes_{\alpha}\Gamma$ is $\mathcal{Z}$-stable. 
Since $A\rtimes_{\alpha}\Gamma$ can be regarded as a subalgebra of 
$A\otimes \mathbb{K}(\ell^2(\Gamma))$, $A\rtimes_{\alpha}\Gamma$ is stably projectionless. 
Hence it is enough to show that we have 
$
\tau_p (e_{u})=\tau_p (e_{\alpha})
$
for any $p\in P_{\lambda_{\tilde{\alpha}}, \mu_{\tilde{\alpha}}}$ by Proposition \ref{pro:tracial-state}, 
Proposition \ref{pro:Jones-3.1.2} and \cite[Corollary 2.10]{Na0}. 
Since we have 
\begin{align*}
\tau_{p}(e_u)
&=\frac{1}{X_p(\iota)|\Gamma|}\sum_{h\in N(\tilde{\alpha})}X_p(h)
\tilde{\tau}_{A}(\pi_{\tau_A}(u_h)V_h) 
\end{align*}
and 
\begin{align*}
\tau_{p}(e_\alpha)
&=\frac{1}{X_p(\iota)|\Gamma|}\sum_{h\in N(\tilde{\alpha})}X_p(h)
\tilde{\tau}_{A}(V_h), 
\end{align*}
we obtain the conclusion by the assumption. 
\end{proof}

As an application of the proposition above, we shall show the following proposition. 

\begin{pro}\label{pro:def-w-absorbing}
Let $\alpha$ be an outer $\mathcal{W}$-absorbing action of a finite group $\Gamma$ on 
$\mathcal{W}$, that is, there exist a C$^*$-algebra $B$ 
and an action $\beta$ of $\Gamma$ on $B$ such that $\alpha$ is cocycle conjugate to  
$\beta\otimes \mathrm{id}_{\mathcal{W}}$ on $B\otimes \mathcal{W}$. 
Then there exists an outer action $\gamma$ of $\Gamma$ 
on $\mathcal{W}$ such that $\alpha$ is conjugate to 
$\gamma\otimes\mathrm{id}_{\mathcal{W}}$ on 
$\mathcal{W}\otimes \mathcal{W}$. 
\end{pro}
\begin{proof}
By the assumption, there exist an isomorphism $\theta$ from $\mathcal{W}$ onto 
$B\otimes\mathcal{W}$ and a $\beta\otimes \mathrm{id}_{\mathcal{W}}$-cocycle $u$ such that 
$\theta\circ \alpha_g =\mathrm{Ad}(u_g)\circ (\beta_g\otimes\mathrm{id}_{\mathcal{W}})\circ \theta$ 
for any $g\in \Gamma$. 
Take an isomorphism $\varphi$ from $\mathcal{W}$ onto $\mathcal{W}\otimes \mathcal{W}$ and 
a map $V$ from $N(\tilde{\beta})$ to $U(\pi_{\tau_{B}}(B)^{''})$ such that 
$\tilde{\beta_h}=\mathrm{Ad}(V_h)$. 
Define an action $\beta^{\prime}$ on $B\otimes \mathcal{W}\otimes \mathcal{W}$ by 
$$
\beta^{\prime}_g:= (\mathrm{id}_{B}\otimes \varphi)\circ \mathrm{Ad}(u_g) \circ 
(\beta_g\otimes\mathrm{id}_{\mathcal{W}})\circ (\mathrm{id}_B\otimes \varphi )^{-1}
$$ 
for any 
$g\in \Gamma$. 
Then $\alpha$ is conjugate to $\beta^{\prime}$. 
Note that $\beta^{\prime}$ is outer because $\alpha$ is outer. 
It is easy to see that we have 
$$
\tilde{\beta}^{\prime}_h= \mathrm{Ad}(\pi_{\tau_{B\otimes \mathcal{W}\otimes\mathcal{W}}}
((\mathrm{id}_{B}\otimes \varphi)(u_h))(V_h\otimes 1_{\pi_{\tau_{\mathcal{W}}}(\mathcal{W})^{''}}\otimes 
1_{\pi_{\tau_{\mathcal{W}}}(\mathcal{W})^{''}}))
$$ 
for any $h\in N(\tilde{\beta}^{\prime})$. 
Put 
$$
w_g:= (u_g\otimes 1_{M(\mathcal{W})})(\mathrm{id}_B\otimes \varphi)(u_g^*)\in 
U(M(B\otimes \mathcal{W}\otimes \mathcal{W}))
$$
for any $g\in \Gamma$. 
Then $w$ is a $\beta^{\prime}$-cocycle and 
$$
\mathrm{Ad}(w_g)\circ \beta^{\prime}_g =  (\mathrm{Ad}(u_g)\circ 
(\beta_g\otimes \mathrm{id}_{\mathcal{W}})) \otimes \mathrm{id}_{\mathcal{W}}\quad 
\text{on}\quad B\otimes \mathcal{W}\otimes \mathcal{W}
$$
for any $g\in \Gamma$. Define an action $\gamma$ of $\Gamma$ on 
$\mathcal{W}$ 
by 
$$
\gamma_g:= \theta^{-1}\circ 
\mathrm{Ad}(u_g)\circ (\beta_g\otimes \mathrm{id}_{\mathcal{W}})\circ \theta 
$$
for any $g\in \Gamma$. Then $\mathrm{Ad}(w)\circ \beta^{\prime}$ is conjugate to 
$\gamma\otimes \mathrm{id}_{\mathcal{W}}$ on $\mathcal{W}\otimes \mathcal{W}$. 
We have 
\begin{align*}
&\tilde{\tau}_{B\otimes \mathcal{W}\otimes \mathcal{W}}
(\pi_{\tau_{B\otimes \mathcal{W}\otimes\mathcal{W}}}(w_h)
\pi_{\tau_{B\otimes \mathcal{W}\otimes\mathcal{W}}}
((\mathrm{id}_{B}\otimes \varphi)(u_h))(V_h\otimes 1_{\pi_{\tau_{\mathcal{W}}}(\mathcal{W})^{''}}\otimes 
1_{\pi_{\tau_{\mathcal{W}}}(\mathcal{W})^{''}})) \\
&=\tilde{\tau}_{B\otimes \mathcal{W}\otimes \mathcal{W}}
(\pi_{\tau_{B\otimes \mathcal{W}\otimes\mathcal{W}}}(u_h\otimes 1_{M(\mathcal{W})})
(V_h\otimes 1_{\pi_{\tau_{\mathcal{W}}}(\mathcal{W})^{''}}\otimes 1_{\pi_{\tau_{\mathcal{W}}}(\mathcal{W})^{''}}))
\\
&= \tilde{\tau}_{B\otimes \mathcal{W}\otimes \mathcal{W}}
(\pi_{\tau_{B\otimes \mathcal{W}}}(u_h)(V_h \otimes 1_{\pi_{\tau_{\mathcal{W}}}(\mathcal{W})^{''}})
\otimes 1_{\pi_{\tau_{\mathcal{W}}}(\mathcal{W})^{''}}) \\
&= \tilde{\tau}_{B\otimes \mathcal{W}}
(\pi_{\tau_{B\otimes \mathcal{W}}}(u_h)(V_h\otimes 1_{\pi_{\tau_{\mathcal{W}}}(\mathcal{W})^{''}})) \\
&= \tilde{\tau}_{B\otimes \mathcal{W}\otimes \mathcal{W}}((\widetilde{\mathrm{id}_{B}\otimes \varphi})
(\pi_{\tau_{B\otimes \mathcal{W}}}(u_h)(V_h\otimes 1_{\pi_{\tau_{\mathcal{W}}}(\mathcal{W})^{''}}))) \\
&= \tilde{\tau}_{B\otimes \mathcal{W}\otimes \mathcal{W}}
(\pi_{\tau_{B\otimes \mathcal{W}\otimes\mathcal{W}}}((\mathrm{id}_{B}\otimes \varphi)(u_h))
(V_h\otimes 1_{\pi_{\tau_{\mathcal{W}}}(\mathcal{W})^{''}}\otimes 1_{\pi_{\tau_{\mathcal{W}}}(\mathcal{W})^{''}}))
\end{align*}
for any $h\in N(\tilde{\beta}^{\prime})$ where $\widetilde{\mathrm{id}_{B}\otimes \varphi}$ is the 
induced isomorphism from $\pi_{\tau_{B\otimes \mathcal{W}}}(B\otimes\mathcal{W})^{''}$ onto 
$\pi_{\tau_{B\otimes \mathcal{W}\otimes\mathcal{W}}}(B\otimes\mathcal{W}\otimes\mathcal{W})^{''}$
by $\mathrm{id}_{B}\otimes \varphi$. Hence Proposition \ref{pro:coboundary} implies that 
$\beta^{\prime}$ is conjugate to $\mathrm{Ad}(w)\circ \beta^{\prime}$.  
Consequently, $\alpha$ is conjugate to $\gamma\otimes \mathrm{id}_{\mathcal{W}}$ on 
$\mathcal{W}\otimes \mathcal{W}$. 
\end{proof}

\subsection{Kirchberg's relative central sequence C$^*$-algebras}

We shall recall Kirchberg's relative central sequence C$^*$-algebras in \cite{Kir2}. 
Fix a free ultrafilter $\omega$ on $\mathbb{N}$. For a C$^*$-algebra $A$, put
$$
A^{\omega}:= \ell^{\infty}(\mathbb{N}, A)/ \{\{x_n\}_{n\in\mathbb{N}}\; |\; \lim_{n\to\omega} \|x_n\| =0 \}.
$$
Let $(x_n)_n$ denote a representative of an element in $A^{\omega}$. 
Every action $\alpha$ of a discrete group $\Gamma$ on $A$ induces an action on 
$A^{\omega}$. We denote it by the same symbol $\alpha$ unless otherwise specified. 
For a tracial state $\tau_A$ on $A$, 
we can define a tracial state $\tau_{A, \omega}$ on $A^{\omega}$ such that 
$\tau_{A, \omega}((x_n)_n)=\lim_{n\to\omega}\tau_A(x_n)$ for any $(x_n)_n\in A^{\omega}$. 
For a homomorphism $\Phi$ from a C$^*$-algebra $B$ to $A^{\omega}$, put 
$$
F(\Phi(B), A):= A^{\omega}\cap \Phi(B)^{\prime}/\{(x_n)_n\in A^{\omega}\cap \Phi(B)^{\prime}\; |\; 
(x_n)_n\Phi (b)=0\; \text{for any}\; b\in B\}. 
$$
If $B\subseteq A$ and $\Phi$ is a natural inclusion map, then we denote $F(\Phi (B), A)$ by $F(B, A)$. 
Furthermore, we denote $F(B,A)$ by $F(A)$ if $B=A$. 
If an action $\alpha$ of $\Gamma$ on $A^{\omega}$ satisfies $\alpha_g(\Phi(B))=\Phi(B)$ for any 
$g\in\Gamma$, then $\alpha$ induces an action on $F(\Phi (B), A)$. 
We denote it by the same symbol $\alpha$ unless otherwise specified. 
If an action $\alpha$ of $\Gamma$ on $A$ is cocycle conjugate to an action $\beta$ of 
$\Gamma$ on $B$, then 
$\alpha$ on $F(A)$ is conjugate to $\beta$ on $F(B)$. In particular, 
if $\alpha$ is cocycle conjugate to $\beta$, then $F(A)^{\alpha}$ is isomorphic to 
$F(B)^{\beta}$. 
If $\tau_A$ is a tracial state on $A$ such that $\tau_{A, \omega}\circ \Phi$ is a state on $B$, then 
$\tau_{A, \omega}$ induces a tracial state on $F(\Phi(B), A)$ by \cite[Proposition 2.1]{Na4}. 
We denote it by the same symbol $\tau_{A, \omega}$. 

Let $A$ and $B$ be simple separable monotracial C$^*$-algebras, and 
let $\Phi$ be a trace preserving homomorphism from $B$ to $A$. Put
$$
\mathcal{M}^{\omega}(A):=\ell^{\infty}(\mathbb{N}, \pi_{\tau_{A}}(A)^{''})/ 
\{ \{x_n\}_{n=1}^{\infty}\in \ell^{\infty}(\mathbb{N}, \pi_{\tau_{A}}(A)^{''})\; |\; \lim_{n\to \omega}\| x_n\|_2
=0 \}.
$$
We also denote by $(x_n)_n$ a representative of an element in $\mathcal{M}^{\omega}(A)$. 
Define a homomorphism $\varrho$ from $A^{\omega}$ to $\mathcal{M}^{\omega}(A)$ by 
$\varrho ((x_n)_n)= (\pi_{\tau_A}(x_n))_n$ for any $(x_n)_n\in A^{\omega}$. 
Then the Kaplansky density theorem implies that $\varrho$ is surjective. 
Put 
$$
\mathcal{M}_{\omega}(\Phi(B), A):= \mathcal{M}^{\omega}(A)\cap \varrho (\Phi (B))^{\prime}.
$$
By an argument after \cite[Proposition 3.7]{Na5} (see also \cite[Theorem 3.3]{KR} and 
\cite[Theorem 3.1]{MS3}), we see that $\varrho$ induces a surjective homomorphism from 
$F(\Phi(B), A)$ onto $\mathcal{M}_{\omega}(\Phi(B), A)$. We denote it by the same symbol 
$\varrho$ for simplicity. 
If $B=A$ and $\Phi(x)=(x)_n$ for any $x\in A$, then we denote $\mathcal{M}_{\omega}(\Phi(B), A)$ 
by $\mathcal{M}_{\omega}(A)$. Note that $\mathcal{M}_{\omega}(A)$ is nothing but 
the von Neumann algebraic central sequence algebra of $\pi_{\tau_{A}}(A)^{''}$. 
If an action $\alpha$ of $\Gamma$ on $A^{\omega}$ satisfies $\alpha_g(\Phi(B))=\Phi(B)$ for any 
$g\in\Gamma$, then $\alpha$ induces an action on $\mathcal{M}_{\omega}(\Phi(B), A)$. 
We denote it by $\tilde{\alpha}$ (the same as on $\pi_{\tau_{A}}(A)^{''}$). 

\subsection{Approximate cocycle morphisms and property W}
We shall recall some notions in \cite{Na5} and \cite{Na7}. 
Let $A$ and $B$ be separable C$^*$-algebras, and let $\alpha$ and $\beta$ 
be actions of a discrete 
group $\Gamma$ on $A$ and $B$, respectively. 
We say that a pair $(\Psi, u)$ is a \textit{sequential asymptotic cocycle morphism from 
$(A, \alpha)$ to $(B,\beta)$} if $\Psi$ is a homomorphism from $A$ to $B^{\omega}$ and 
$u$ is a map from $\Gamma$ to $U((B^{\sim})^{\omega})$ such that 
$$
\Psi \circ \alpha_g(x)= \mathrm{Ad}(u_g)\circ \beta_g \circ \Psi (x) \quad \text{and} \quad 
\Psi (x) u_{gh}= \Psi (x) u_g\beta_g(u_h)
$$
for any $x\in A$ and $g, h\in \Gamma$. 
A sequential asymptotic cocycle morphism $(\Psi, u)$ from $(A, \alpha)$ to $(A, \alpha)$ 
is said to be \textit{inner} if there exists an element $w$ in $U((A^{\sim})^{\omega})$ such that 
$$
\Psi (x)= wxw^* \quad \text{and} \quad \Psi(x)u_g =\Psi (x) w\alpha_g(w^*)
$$
for any $x\in A$ and $g\in\Gamma$. 
For $\Gamma_0\subset \Gamma$, $F\subset A$ and $\varepsilon>0$, 
we say that a pair $(\varphi, u)$ is a \textit{proper $(\Gamma_0, F, \varepsilon)$-approximate 
cocycle morphism from $(A, \alpha)$ to $(B, \beta)$} if $\varphi$ is a 
completely positive map from $A$ to 
$B$ and $u$ is a map from $\Gamma$ to $U(B^{\sim})$ such that 
$$
u_{\iota}=1_{B^{\sim}}, \quad 
\| \varphi (xy)- \varphi(x)\varphi (y)\| <\varepsilon, \quad 
\| \varphi (x) (u_{gh}-u_g\beta_g (u_h))\|<\varepsilon, 
$$
and
$$
\| \varphi (\alpha_g (x))- u_g\beta_g(\varphi (x))u_g^*\| <\varepsilon
$$
for any $g\in \Gamma_0$ and $x,y\in F$. 
Let $\{\varepsilon_n\}_{n=1}^{\infty}$ be a sequence of positive real numbers such that 
$\lim_{n\to\infty}\varepsilon_n=0$, $\{\Gamma_n\}_{n=1}^{\infty}$ a sequence of finite subsets of 
$\Gamma$ such that $\bigcup_{n=1}^{\infty} \Gamma_n=\Gamma$ and $\{F_n\}_{n=1}^{\infty}$ 
a sequence of finite subsets of $A$ such that $\overline{\bigcup_{n=1}^{\infty}F_n}=A$. 
If $(\varphi_n, u_n)$ is a proper $(\Gamma_n, F_n, \varepsilon_n)$-approximate cocycle 
morphism from $(A, \alpha)$ to $(B, \beta)$, then $\{(\varphi_n, u_n)\}_{n=1}^{\infty}$ induces 
a sequential asymptotic cocycle morphism from $(A, \alpha)$ to $(B, \beta)$. 

\begin{Def} (cf. \cite[Definition 2.2]{Na7}.) \ \\
We say that an action  $\gamma$ of a countable discrete group $\Gamma$ on $\mathcal{W}$ 
has \textit{property W} if $\gamma$ satisfies the following properties: \ \\
(i) for any $k\in\mathbb{N}$, there exists a unital homomorphism from 
$M_{k}(\mathbb{C})$ to $F(\mathcal{W})^{\gamma}$, \ \\
(ii) if $x$ and $y$ are normal elements in $F(\mathcal{W})^{\gamma}$ such that 
$\mathrm{Sp}(x)=\mathrm{Sp}(y)$ and $0<\tau_{\mathcal{W}, \omega}(f(x))=
\tau_{\mathcal{W}, \omega}(f(y))$ for any $f\in C(\mathrm{Sp}(x))_{+}\setminus \{0\}$, then 
$x$ and $y$ are unitarily equivalent in $F(\mathcal{W})^{\gamma}$. 
\end{Def}

We have the following theorem. Note that the proof of this theorem is based on techniques around 
equivariant property (SI) (see, for example,  \cite{MS}, \cite{MS3}, \cite{MS2},  \cite{Sa0}, 
\cite{Sa}, \cite{Sa2} and \cite{Sza6}).

\begin{thm}\label{thm:property-W}
Let $\gamma$ be an outer action of a countable amenable group on $\mathcal{W}$. 
\ \\
(1) If $\gamma$ is a $\mathcal{W}$-absorbing action, then  
$\gamma$ has property W. \ \\
(2) Assume that $\gamma$ has property W. If $h$ is a positive element in $F(\mathcal{W})^{\gamma}$ 
satisfying 
$d_{\tau_{\mathcal{W}, \omega}}(h)>0$, then for any $\theta\in [0, d_{\tau_{\mathcal{W}, \omega}}(h))$, 
there exists a projection $p_{\theta}$ in $\overline{hF(\mathcal{W})^{\gamma}h}$ such that 
$\tau_{\mathcal{W}, \omega}(p_{\theta})=\theta$. 
\ \\
(3) Assume that $\gamma$ has property W.  If $p$ and $q$ are projections in 
$F(\mathcal{W})^{\gamma}$ such that $0<\tau_{\mathcal{W}, \omega}(p)
=\tau_{\mathcal{W}, \omega}(q)$, then $p$ and $q$ are Murray-von Neumann equivalent. 
\end{thm}
\begin{proof}
(1) and (3) are immediate consequences of \cite[Proposition 4.2.(i) and Theorem 4.5]{Na5} and 
\cite[Theorem 2.3.(iii)]{Na7}. We shall show (2). By \cite[Theorem 2.3.(i)]{Na7}, 
there exists a projection $q_{\theta}$ in $F(\mathcal{W})^{\gamma}$ such that 
$\tau_{\mathcal{W}, \omega}(q_{\theta})=\theta$. 
\cite[Proposition 3.11 and Proposition 3.12]{Na5} imply that there exists an element $r$ in 
$F(\mathcal{W})^{\gamma}$ such that $r^*hr=q_{\theta}$. Put $p_{\theta}:= h^{1/2}rr^*h^{1/2}$. 
Then $p_{\theta}$ is the desired projection.  
\end{proof}

\begin{pro}\label{pro:small-projection}
Let $\gamma$ be an outer action of a countable amenable group $\Gamma$ on $\mathcal{W}$, 
and let $F$ be a finite subset of $\mathrm{ker}\; \varrho|_{F(\mathcal{W})}$. 
If $\gamma$ has property W, then there exists a projection $p$ in 
$\mathrm{ker}\;\varrho|_{F(\mathcal{W})^{\gamma}}$ such that $px=xp=x$ for any $x\in F$. 
\end{pro}
\begin{proof}
Put $f:=\sum_{x\in F}(x^*x+xx^*)\in \mathrm{ker}\; \varrho|_{F(\mathcal{W})}$. 
It suffices to show that there exists a projection 
$p$ in $\mathrm{ker}\;\varrho|_{F(\mathcal{W})^{\gamma}}$ such that $pf=f$. 
By \cite[Lemma 1.4]{Kas}, there exists a positive contraction 
$h$ in $\mathrm{ker} \varrho|_{F(\mathcal{W})^{\gamma}}$ such that $hf=f$.  
Since we have $d_{\tau_{\mathcal{W}, \omega}}(1_{F(\mathcal{W})}-h)=1$, for any 
$0<\varepsilon<1$, there exists a 
projection $q_{\varepsilon}$ in $\overline{(1_{F(\mathcal{W})}-h)F(\mathcal{W})^{\gamma}
(1_{F(\mathcal{W})}-h)}$ such that 
$\tau_{\mathcal{W}, \omega}(q_{\varepsilon})=1-\varepsilon$ by 
Theorem \ref{thm:property-W}. Note that we have 
$q_{\varepsilon}f=0$ for any $0<\varepsilon<1$. By a diagonal argument, we see that there exists a 
projection $q$ in $F(\mathcal{W})^{\gamma}$ such that $\tau_{\mathcal{W}, \omega}(q)=1$ and 
$qf=0$. Put $p:=1_{F(\mathcal{W})}-q$. Then we obtain the conclusion. 
\end{proof}

\section{Model actions}\label{sec:mod}

In this section we shall construct model actions. 
First, we shall construct model actions in the classification for outer 
$\mathcal{W}$-absorbing actions up to cocycle conjugacy.  
The following lemma is a generalization of \cite[Lemma 3.2]{Na6}.

\begin{lem}\label{lem:key-model}
Let $\Gamma$ be a finite group. For any real number $r$ with $0<r<1$, there exist a simple 
unital monotracial approximately finite-dimensional (AF) algebra $A$ and a unitary representation 
$V$ of $\Gamma$ on $\pi_{\tau_A}(A)^{''}$ such that 
$\mathrm{Ad}(V)$ induces an outer action of $\Gamma$ 
on $A$ and 
$\tilde{\tau}_A(V_g)=r$ for any $g\in \Gamma\setminus\{\iota\}$. 
\end{lem}
\begin{proof}
By the Effros-Handelman-Shen theorem \cite{EHS} (or \cite[Theorem 2.2]{Ell-order}), 
there exists a simple unital monotracial AF algebra $B$ such that 
$$
K_0(B)=\left\{a_0+\sum_{n=1}^{m}a_n r^{\frac{1}{2^n}}\in\mathbb{R}\; |\; m\in\mathbb{N}, a_0,a_1,...,a_m\in
\mathbb{Z} \right\},
$$
$K_0(B)_{+}=K_0(B)\cap\mathbb{R}_{+}$ and $[1_{B}]_0=1$. 
For any $n\in\mathbb{N}$, there exists a projection $p_n$ in $B$ such that 
$\tau_{B}(p_n)=r^{\frac{1}{2^n}}$. For any  $g\in \Gamma$ and $n\in\mathbb{N}$, 
put
$$
u_{g, n}:= p_n\otimes 1_{ \mathbb{K}(\ell^2(\Gamma))}
 + (1_{B}-p_n)\otimes \lambda_g^{\Gamma}\in B\otimes \mathbb{K}(\ell^2(\Gamma))
$$
where $\lambda^{\Gamma}$ is the left regular representation of $\Gamma$. 
Then $u_{g,n}$ is a unitary element such that 
$\tau_{B\otimes \mathbb{K}(\ell^2(\Gamma))}(u_{g,n})=r^{\frac{1}{2^n}}$ for any 
$g\in \Gamma\setminus\{\iota\}$ and $n\in\mathbb{N}$ 
because we have $\tau_{\mathbb{K}(\ell^2(\Gamma))}(\lambda^{\Gamma}_g)=0$ for any 
$g\in \Gamma\setminus \{\iota\}$. 
Put 
$$
A:=\bigotimes_{n=1}^{\infty} B\otimes \mathbb{K}(\ell^2(\Gamma)).
$$
Then $A$ is a simple unital monotracial AF algebra. 
For any $g\in \Gamma$ and $n\in\mathbb{N}$, put
$$
w_{g, n}:= u_{g, 1}\otimes u_{g, 2}\otimes \cdots \otimes u_{g, n}\otimes 
1_{B\otimes \mathbb{K}(\ell^2(\Gamma))}\otimes \cdots \in 
\bigotimes_{n=1}^{\infty} B\otimes \mathbb{K}(\ell^2(\Gamma))=A.
$$
Since we have 
\begin{align*}
\| w_{g, m}- w_{g,n}\|_2^2
&= 2- 2\mathrm{Re}\tau_{A}(1\otimes \cdots 1\otimes u_{g, n+1}
\otimes \cdots \otimes u_{g, m}\otimes 1\cdots ) \\
&= 2- 2\prod_{j=n+1}^{m}r^{\frac{1}{2^j}}=2-2r^{\sum_{j=n+1}^m\frac{1}{2^j}}
\end{align*}
for any $m>n$, the same argument as in the proof of \cite[Lemma 3.2]{Na6} 
shows that for any $g\in \Gamma$, there exists a unitary element $V_g$ in 
$\pi_{\tau_A}(A)^{''}$ such that 
$\{\pi_{\tau_A}(w_{g,n})\}_{n\in\mathbb{N}}$ converges to $V_{g}$ in the strong-$^*$ topology. 
It is easy to see that the map $V$ defined by $\Gamma\ni g\mapsto V_g\in 
U(\pi_{\tau_A}(A)^{''})$ is a unitary representation of $\Gamma$ on $\pi_{\tau_A}(A)^{''}$ such that 
$\tilde{\tau}_A(V_g)=r$ for any $g\in \Gamma\setminus\{\iota\}$. 
Furthermore, it can be easily checked that $\mathrm{Ad}(V)$ induces an action 
$\alpha$ of $\Gamma$ on $A$. 
Indeed, we have 
$$
\alpha_g(x)= \bigotimes_{n=1}^{\infty} \mathrm{Ad}(w_{g,n})(x) 
$$
for any $x\in A$ and $g\in\Gamma$. For any $g\in \Gamma$, put
$$
e_g:= [(\overbrace{1_{B\otimes \mathbb{K}(\ell^2(\Gamma))}\otimes \cdots \otimes 
1_{B\otimes \mathbb{K}(\ell^2(\Gamma))}}^{n-1}\otimes 
(1_{B}\otimes e_{g,g}^{\Gamma})\otimes 1_{B\otimes \mathbb{K}(\ell^2(\Gamma))} \otimes \cdots)_n]
\in F(A).
$$ 
If $g\in\Gamma\setminus\{\iota\}$, then we have $\| (1_{B}-p_n)\otimes (e_{g^2,g^2}-e_{g,g})\|=
\| e_{g^2,g^2}-e_{g,g}\|=1$ for any $n\in\mathbb{N}$.  
Hence 
\begin{align*}
\alpha_g(e_g)-e_g 
= [(1_{B\otimes \mathbb{K}(\ell^2(\Gamma))}\otimes \cdots \otimes  
((1_B-p_n)\otimes (e_{g^2,g^2}-e_{g,g}))\otimes \cdots)_n]\neq 0
\end{align*}
for any $g\in\Gamma\setminus\{\iota\}$. 
Therefore $\alpha_g$ is not a trivial automorphism of $F(A)$ for any 
$g\in\Gamma\setminus\{\iota\}$. 
Consequently, $\alpha$ is an outer action on $A$. 
\end{proof}

Let $M_{n^{\infty}}$ denote the uniformly hyperfinite (UHF) algebra of type $n^{\infty}$. 
The following proposition is an analogous proposition of \cite[Proposition 1.5.8]{Jones}.

\begin{pro}\label{model:cocycle-conjugacy}
Let $\Gamma$ be a finite group, and let $N$ be a normal subgroup of $\Gamma$. 
For any $\Lambda\in \Lambda(\Gamma, N)$, there exists a simple unital monotracial AF algebra 
$A_{(\Gamma, N, \Lambda)}$ and an outer action $S^{(\Gamma, N, \Lambda)}$ of 
$\Gamma$ on $A_{(\Gamma, N, \Lambda)}$ such that
$$
N(\tilde{S}^{(\Gamma, N, \Lambda)})=N \quad \text{and} \quad 
\Lambda (\tilde{S}^{(\Gamma, N, \Lambda)})=
\Lambda.
$$ 
\end{pro}
\begin{proof}
Choose a pair $(\lambda, \mu)\in Z(\Gamma, N)$ such that $[(\lambda, \mu)]=\Lambda$. 
Define an action $\gamma$ of $\Gamma$ on $M_{|\Gamma|^{\infty}}\cong \bigotimes_{n=1}^{\infty}
\mathbb{K}(\ell^2(\Gamma))$ by 
$$
\gamma_g:= \bigotimes_{n=1}^{\infty} \mathrm{Ad}(\lambda_g^{\Gamma})
$$
for any $g\in \Gamma$ where $\lambda^{\Gamma}$ is the left regular representation of $\Gamma$. 
Let $B$ be the (reduced) twisted crossed product C$^*$-algebra 
$M_{|\Gamma|^{\infty}}\rtimes_{\gamma|_{N}, \mu}N$. 
Since we have $N(\tilde{\gamma}|_{N})=\{\iota\}$ (see, for example, \cite[Section 1.5]{Jones}), 
$M_{|\Gamma|^{\infty}}\rtimes_{\gamma|_{N}, \mu}N$ 
is simple and monotracial. (See, for example, \cite[Theorem 1]{Bed} and \cite[Theorem 3.2]{K}.)
Furthermore, $B$ is a unital AF algebra because $\gamma|_{N}$ is a product type action. Define an 
action $\beta$ of $\Gamma$ on $B$ by
$$
\beta_g\left( \sum_{h\in N}x_h\lambda^{\gamma}_h\right)= \sum_{h\in N}
\lambda(g, h)\gamma_g(x_{g^{-1}hg})\lambda^{\gamma}_{h}
$$
for any  $\sum_{h\in N}x_h\lambda^{\gamma}_h\in B$ and $g\in\Gamma$. 
Note that we have $N(\beta)=N(\tilde{\beta})=N$. 
By Lemma \ref{lem:key-model}, there exists a simple unital monotracial AF algebra $A$ and 
a unitary representation $V$ of $\Gamma$ on $\pi_{\tau_A}(A)^{''}$ such that 
$\mathrm{Ad}(V)$ induces an outer action $\alpha$ of $\Gamma$ on $A$. Set
$$
A_{(\Gamma, N, \Lambda)}:= A\otimes B \quad \text{and} \quad S^{(\Gamma, N, \Lambda)}
:= \alpha\otimes\beta.
$$
Then $A_{(\Gamma, N, \Lambda)}$ is a simple unital monotracial AF algebra and we have 
$N(\tilde{S}^{(\Gamma, N, \Lambda)})=N$ and $N(S^{(\Gamma, N, \Lambda)})=\{\iota\}$. 
Define a map $V^{\prime}$ from $N$ to 
$\pi_{\tau_{A_{(\Gamma, N, \Lambda)}}}(A_{(\Gamma, N, \Lambda)})^{''}\cong 
\pi_{\tau_A}(A)^{''}\bar{\otimes}\pi_{\tau_B}(B)^{''}$ 
by 
$$
V^{\prime}_h:= V_h\otimes \pi_{\tau_B}(\lambda_h^{\gamma})
$$
for any $h\in N$. Then $V^{\prime}$ is a map from $N$ to 
$U(\pi_{\tau_{A_{(\Gamma, N, \Lambda)}}}(A_{(\Gamma, N, \Lambda)})^{''})$ such that 
$\tilde{S}^{(\Gamma, N, \Lambda)}_h=\mathrm{Ad}(V^{\prime}_h)$ for any $h\in N$. 
It can be easily checked that
$$
\tilde{S}^{(\Gamma, N, \Lambda)}_g(V^{\prime}_{g^{-1}hg})=\lambda (g, h)V^{\prime}_{h} \quad 
\text{and} \quad V_{h}^{\prime}V_{k}^{\prime}= \mu (h,k)V_{hk}^{\prime} 
$$ 
for any $g\in \Gamma$ and $h,k\in N$. Hence we have $\Lambda(\tilde{S}^{(\Gamma, N, \Lambda)})
=[(\lambda, \mu)]=\Lambda$. 
\end{proof}

We need the following corollary in order to prove the classification theorem 
for outer $\mathcal{W}$-absorbing actions up to conjugacy. 

\begin{cor}\label{cor:1.5.9} (Cf. \cite[Proposition 1.5.9]{Jones}.) \ \\
With notation as above, let 
$C:= A_{(\Gamma, N, \Lambda)}\otimes \mathcal{W}\otimes \mathbb{K}(\ell^2(\Gamma))$ and 
$\alpha^{\prime}:= S^{(\Gamma, N, \Lambda)}\otimes \mathrm{id}_{\mathcal{W}}\otimes 
\mathrm{Ad}(\rho)$ where $\rho$ is the right regular representation of $\Gamma$ 
on $\ell^2(\Gamma)$. 
Suppose that $W$ is a map from $N$ to $U(\pi_{\tau_{C}}(C)^{''})$ such that 
$\tilde{\alpha}^{\prime}_h=\mathrm{Ad}(W_h)$ for any $h\in N$. 
Then, for any $\eta\in H^1(N)^{\Gamma}$, there exists an automorphism $\theta$ of 
$C$ 
such that 
$$
\theta \circ \alpha^{\prime}_g=\alpha^{\prime}_g\circ \theta \quad  
\text{and} 
\quad
\tilde{\theta}(W_h)=\eta (h)W_h
$$ 
for any $g\in \Gamma$ and $h\in N$. 
\end{cor}
\begin{proof}
With notation as above (in the proof of Proposition \ref{model:cocycle-conjugacy}), 
there exists a map $\zeta$ from $N$ to $\mathbb{T}$ such that 
$W_h= \zeta (h)(V_h\otimes \pi_{\tau_B}(\lambda^{\gamma}_h)\otimes 
1_{\pi_{\tau_{\mathcal{W}}}(\mathcal{W})^{''}}
\otimes \rho_h)$ for any $h\in N$ 
because $\pi_{\tau_{C}}(C)^{''}$ is a factor.  
Since $\eta$ is a group homomorphism from $N$ to $\mathbb{T}$, 
we can define an automorphism $\sigma$ of 
$B=M_{|\Gamma|^{\infty}}\rtimes_{\gamma|_{N}, \mu}N$ 
by 
$$
\sigma \left( \sum_{h\in N}x_h\lambda^{\gamma}_h\right)= \sum_{h\in N}\eta(h)x_h\lambda^{\gamma}_h
$$
for any  $\sum_{h\in N}x_h\lambda^{\gamma}_h\in B$. 
Note that we have $\sigma \circ \beta_g= \beta_g\circ \sigma$ for any $g\in \Gamma$ 
because we have $\eta (g^{-1}hg)=\eta (h)$ for any $g\in\Gamma$ and $h\in N$. 
Hence 
$\theta:= \mathrm{id}_A\otimes \sigma\otimes 
\mathrm{id}_{\mathcal{W}\otimes\mathbb{K}(\ell^2(\Gamma))}$ 
is an automorphism of $C$ 
such that $\theta \circ \alpha_g^{\prime}
=\alpha_g^{\prime} \circ \theta$ for any $g\in \Gamma$. 
Furthermore, we have 
\begin{align*}
\tilde{\theta}(W_h)
&= \zeta(h) (V_h\otimes \pi_{\tau_{B}}(\sigma (\lambda^{\gamma}_h))\otimes 
1_{\pi_{\tau_{\mathcal{W}}}(\mathcal{W})^{''}}\otimes \rho_h) \\
& =\zeta(h)(V_h \otimes \eta(h)\pi_{\tau_B}(\lambda^{\gamma}_h)\otimes 
1_{\pi_{\tau_{\mathcal{W}}}(\mathcal{W})^{''}}\otimes \rho_h) \\
&=\eta(h)\zeta(h)
(V_h \otimes \pi_{\tau_B}(\lambda^{\gamma}_h)\otimes 
1_{\pi_{\tau_{\mathcal{W}}}(\mathcal{W})^{''}}\otimes \rho_h)=\eta(h)W_h
\end{align*}
for any $h\in N$. Therefore we obtain the conclusion. 
\end{proof}

In the rest of this section, we shall construct outer actions $\alpha$ with arbitrary invariants 
$(N(\tilde{\alpha}), \Lambda(\tilde{\alpha}), i(\tilde{\alpha}))$ in the classification for 
outer $\mathcal{W}$-absorbing actions up to conjugacy. 
We refer the reader to \cite[Section 3]{Na6} for the idea of the construction. 
We say that a probability measure $m$ on a finite set $P$ has \textit{full support} if $m(p)>0$ 
for any $p\in P$. 
\begin{pro}(Cf. \cite[Proposition 1.7]{Na6}) \ \\
Let $A$ be a simple separable monotracial C$^*$-algebra, and let $\alpha$ be an outer action 
of a finite group on $A$. If $m$ is a probability measure on 
$P_{\lambda_{\tilde{\alpha}}, \mu_{\tilde{\alpha}}}$ such that 
$i(\tilde{\alpha})=[m]$, then $m$ has full support. 
\end{pro}
\begin{proof}
By (1.2.1), Proposition \ref{pro:tracial-state} and the definition of $i(\tilde{\alpha})$, we may assume 
$
m(p)= X_p(\iota)|\Gamma|\tau_{p}(e_{\alpha})
$
for any $p\in P_{\lambda_{\tilde{\alpha}}, \mu_{\tilde{\alpha}}}$. 
Since $A\rtimes_{\alpha}\Gamma$ is simple, we have $\tau_{p}(e_{\alpha})>0$. 
Therefore we obtain the conclusion. 
\end{proof}

The following lemma is a generalization of \cite[Lemma 3.1]{Na6}. 
See also \cite[Proposition 1.5.8 and Theorem 1.5.11]{Jones}.

\begin{lem}\label{lem:model}
Let $\Gamma$ be a finite group, and let $N$ be a normal subgroup of $\Gamma$. 
For any $(\lambda, \mu)\in Z(\Gamma, N)$ and $m\in M(P_{\lambda, \mu})$, there exist a simple unital 
monotracial AF algebra $A$, an action $\alpha$ of $\Gamma$ on $A$ and a map $v$ 
from $N$ to $U(A)$ such that 
$$
N(\alpha)=N(\tilde{\alpha})=N, \quad \alpha_h=\mathrm{Ad}(v_h), \quad 
\alpha_g(v_{g^{-1}hg})=\lambda(g, h)v_h, \quad v_{h}v_{k}=\mu (h, k)v_{hk}
$$ 
and 
$$
\tau_A(\Phi_{v}(p))=m(p)
$$
for any $g\in \Gamma$, $h, k\in N$ and $p\in P_{\lambda, \mu}$. 
\end{lem}
\begin{proof}
By the same way as in the proof of Proposition \ref{model:cocycle-conjugacy}, 
we obtain a simple unital monotracial AF algebra 
$B=M_{|\Gamma|^{\infty}}\rtimes_{\gamma|_N, \mu} N$ and an action $\beta$ of $\Gamma$ 
on $B$ such that 
$
\beta_g\left( \sum_{h\in N}x_h\lambda^{\gamma}_h\right)= \sum_{h\in N}
\lambda(g, h)\gamma_g(x_{g^{-1}hg})\lambda^{\gamma}_{h}
$
for any  $\sum_{h\in N}x_h\lambda^{\gamma}_h\in B$ and $g\in\Gamma$. 
Note that we regard $\lambda^{\gamma}$ as a map from $N$ to $U(B)$ such that 
$$
\beta_h
=\mathrm{Ad}(\lambda^{\gamma}_h), \quad  
\beta_g(\lambda^{\gamma}_{g^{-1}hg})=\lambda(g, h)\lambda^{\gamma}_h \quad \text{and} \quad  
\lambda_{hk}^{\gamma}=\mu (h,k)\lambda^{\gamma}_h\lambda^{\gamma}_k
$$ 
for any $g\in \Gamma$ and $h,k\in N$. 
Since we have $\tau_{B}= \tau_{M_{|\Gamma|^{\infty}}}\circ E_{\gamma|_N}$, 
$$
\tau_{B} (\Phi_{\lambda^{\gamma}} (p))= X_{p}(\iota)
$$ 
for any $p\in P_{\lambda, \mu}$. 
By \cite[Proposition 1.3.2 and Lemma 1.5.10]{Jones}, there exists a finite sequence 
$\{ d_p\}_{p\in P_{\lambda, \mu}}$ of $[0,1]$ such that 
$$
m(p)=\frac{d_{p}X_{p}(\iota)}{\sum_{p^{\prime}\in P_{\lambda, \mu}}d_{p^{\prime}}X_{p^{\prime}}(\iota)}
$$
for any $p\in P_{\lambda, \mu}$. 
There exists a simple unital monotracial AF algebra $C$ such that 
$K_0(C)$ is the additive subgroup of $\mathbb{R}$ generated by $\mathbb{Q}$ and 
$\{d_p\}_{p\in P_{\lambda, \mu}}$, $K_0(C)_{+}=K_0(C)\cap \mathbb{R}_{+}$ and $[1_C]_0=1$ 
by the Effros-Handelman-Shen theorem \cite{EHS} (or \cite[Theorem 2.2]{Ell-order}).
For any $p\in P_{\lambda, \mu}$, there exists a projection $e_p$ in $C$ such that 
$\tau_{C}(e_p)=d_p$, and 
put $e_p^{\prime}:= \Phi_{\lambda^{\gamma}}(p)\otimes e_p\in B\otimes C$. 
Then $\{e_{p}^{\prime}\; |\; p\in P_{\lambda, \mu}\}$ 
are mutually orthogonal projections in 
$(B\otimes C)^{\beta\otimes\mathrm{id}_{C}}$. 
Note that we also have $e_{p}^{\prime}(\lambda_h^{\gamma}\otimes 1_C)=
(\lambda_h^{\gamma}\otimes 1_C)e_{p}^{\prime}$ 
for any $p\in P_{\lambda, \mu}$ and $h\in N$. 
Set $e:=  \sum_{p\in P_{\lambda, \mu}}e_p^{\prime}$, and let 
$$
A:=  e(B\otimes C)e, \quad \alpha_g:= (\beta_g \otimes \mathrm{id}_{C})|_{e(B\otimes C)e} \quad 
\text{and} \quad v_h:= (\lambda_h^{\gamma} \otimes 1_{C})e
$$
for any $g\in \Gamma$ and $h\in N$. Then it can be easily checked that 
$(A, \alpha, v)$ satisfies the desired property. 
\end{proof}

The following theorem is a generalization of \cite[Theorem 3.3]{Na6}.

\begin{thm}\label{thm:model}
Let $\Gamma$ be a finite group, and let $N$ be a normal subgroup of $\Gamma$. 
For any $\Lambda\in \Lambda(\Gamma, N)$ and $m\in M(P_{\lambda, \mu})$ with full support, 
there exists a simple unital monotracial AF algebra $A_{\Gamma, N, \Lambda, m}$ and an outer action 
$\alpha^{(\Gamma, N, \Lambda, m)}$
of $\Gamma$ on $A_{\Gamma, N, \Lambda, m}$ such that 
$$
N(\tilde{\alpha}^{(\Gamma, N, \Lambda, m)})=N, \quad  
\Lambda({\tilde{\alpha}^{(\Gamma, N, \Lambda, m)}})= \Lambda \quad \text{and}
\quad  
i(\tilde{\alpha}^{(\Gamma, N, \Lambda, m)})=[m].
$$  
\end{thm}
\begin{proof}
Choose a pair $(\lambda, \mu)\in Z(\Gamma, N)$ such that $[(\lambda, \mu)]=\Lambda$. 
Since $m$ has full support, we have 
$$
\lim_{r\to 1-0}\left(\frac{m(p)-X_p(\iota)}{r}+X_p(\iota)\right) = m(p)>0
$$
for any $p\in P_{\lambda, \mu}$ where $\lim_{r\to 1-0}$ means the left-sided limit at $1$. 
Hence there exists a real number $r$ with $0<r<1$ such that 
$$
\frac{m(p)-X_p(\iota)}{r}+X_p(\iota)>0
$$
for any $p\in P_{\lambda, \mu}$. (Note that $X_p(\iota)$ is a real number for any 
$p\in P_{\lambda, \mu}$ by \cite[Proposition 1.3.2]{Jones} and $P_{\lambda, \mu}$ is a finite set.)
Since we have $\sum_{p\in P_{\lambda, \mu}}X_p(\iota)=1$ by \cite[Proposition 1.3.2]{Jones} and 
$m$ is a probability measure on $P_{\lambda, \mu}$, 
$$
\sum_{p\in P_{\lambda, \mu}}\left(\frac{m(p)-X_p(\iota)}{r}+X_p(\iota)\right)=\frac{1-1}{r}
+1=1.
$$
Therefore the map $m^{\prime}$ from $P_{\lambda, \mu}$ to $[0,1]$ defined by 
$$
m^{\prime}(p):=\frac{m(p)-X_p(\iota)}{r}+X_p(\iota)
$$
for any $p\in P_{\lambda, \mu}$ is a probability measure on $P_{\lambda, \mu}$. 
By Lemma \ref{lem:model}, there exist a simple unital monotracial AF algebra $A$, an action 
$\alpha$ of $\Gamma$ on $A$ and a map $v$ from $N$ to $U(A)$ such that 
$$
N(\alpha)=N(\tilde{\alpha})=N, \quad \alpha_h=\mathrm{Ad}(v_h), \quad 
\alpha_g(v_{g^{-1}hg})=\lambda(g,h)v_h, \quad
v_hv_k=\mu(h ,k)v_{hk}   
$$
and 
$$
\tau_{A}(\Phi_v(p))= m^{\prime} (p)
$$
for any $g\in\Gamma$, $h,k\in N$ and $p\in P_{\lambda, \mu}$. 
Lemma \ref{lem:key-model} implies that there exist a unital simple monotracial AF algebra $B$ and a 
unitary representation $V$ of $\Gamma$ on 
$\pi_{\tau_{B}}(B)^{''}$ such that $\mathrm{Ad}(V)$ induces an outer action 
$\beta$ of $\Gamma$ on $B$ and $\tilde{\tau}_{B}(V_g)=r$ for any $g\in \Gamma\setminus \{\iota\}$. 
Put 
$$
A_{\Gamma, N, \Lambda, m}: = A\otimes B \quad \text{and} \quad 
\alpha^{(\Gamma, N, \Lambda, m)}:= \alpha \otimes \beta. 
$$
Then $A_{\Gamma, N, \Lambda, m}$ is a simple unital monotracial AF algebra and 
$\alpha^{(\Gamma, N, \Lambda, m)}$ is an outer action of $\Gamma$ on 
$A_{\Gamma, N, \Lambda, m}$. 
Define a map $V^{\prime}$ from $N$ to 
$U(\pi_{\tau_{A_{\Gamma, N, \Lambda, m}}}(A_{\Gamma, N, \Lambda, m})^{''})$ by 
$$
V_h^{\prime}:= \pi_{\tau_{A}}(v_h)\otimes V_h
$$
for any $h\in N$. 
Then we have $\tilde{\alpha}^{(\Gamma, N, \Lambda, m)}_h=\mathrm{Ad}(V_h^{\prime})$ 
for any $h\in N$. 
It is easy to see that $N(\tilde{\alpha}^{(\Gamma, N, \Lambda, m)})=N$ and 
$\Lambda({\tilde{\alpha}^{(\Gamma, N, \Lambda, m)}})=\Lambda$. 
We have 
\begin{align*}
\tilde{\tau}_{A_{\Gamma, N, \Lambda, m}}(\Phi_{V^{\prime}}(p))
&= \sum_{h\in N}X_{p}(h)\tilde{\tau}_{A_{\Gamma, N, \Lambda, m}}\left( V^{\prime}_h\right) \\
&=\sum_{h\in N}X_{p}(h)\tau_{A}(v_h)\tilde{\tau}_{B}(V_h) \\
&=X_p(\iota)+ r\sum_{h\in N\setminus\{\iota\}}X_{p}(h)\tau_{A}(v_h) \\
&=X_p(\iota)+ r\left(\sum_{h\in N}X_{p}(h)\tau_{A}(v_h) - X_p(\iota)\right) \\
&=X_p(\iota)+ r\left(\tau_{A}(\Phi_{v}(p)) -X_p(\iota)\right) \\
&=X_p(\iota)+ r(m^{\prime}(p)-X_p(\iota)) \\
&=m(p)
\end{align*}
for any $p\in P_{\lambda, \mu}$. Hence $i(\tilde{\alpha}^{(\Gamma, N, \Lambda, m)})=[m]$.  
\end{proof}

\section{Actions with trivial characteristic invariant and property W}\label{sec:app}

In this section we shall show Theorem \ref{thm:app-rep} for an 
existence type theorem in the next section. 

Let $\gamma$ be an outer action of a finite group $\Gamma$ on $\mathcal{W}$. 
Assume that the characteristic invariant of $\tilde{\gamma}$ is trivial. 
For any $g_0\in \Gamma$, 
let $C(g_0)$ be the centralizer of $g_0$ in $\Gamma$, that is, 
$C(g_0)=\{g\in \Gamma\; |\; g_0g=gg_0\}$.
Note that $C(g_0)$ is a subgroup of $\Gamma$ and the characteristic invariant of 
$\tilde{\gamma}|_{C(g_0)}$ is also trivial. For any $g_0\in \Gamma$, define a homomorphism 
$\Phi_{g_0}$ from $\mathcal{W}$ to $M_2(\mathcal{W})^{\omega}$ by 
$$
\Phi_{g_0} (x):=  \left(\left(\begin{array}{cc}
                         x   &     0    \\ 
                         0   &     \gamma_{g_0}(x)          
 \end{array} \right)\right)_n
$$
for any $x\in\mathcal{W}$. 
Since we have $\gamma_g\otimes \mathrm{id}_{M_2(\mathbb{C})} (\Phi_{g_0}(\mathcal{W}))
=\Phi_{g_0}(\mathcal{W})$ for any $g \in C(g_0)$, 
$\gamma|_{C(g_0)}$ induces an action of $C(g_0)$ on $F(\Phi_{g_0}(\mathcal{W}), 
M_2(\mathcal{W}))$. We denote it by $\beta$. 

\begin{pro}\label{pro:strict-comparison-central}
With notation as above, if $a$ and $b$ are positive elements in 
$F(\Phi_{g_0}(\mathcal{W}), M_2(\mathcal{W}))^{\beta}$ satisfying 
$
d_{\tau_{M_2(\mathcal{W}), \omega}}(a)< 
d_{\tau_{M_2(\mathcal{W}), \omega}}(b)
$, 
then there exists an element 
$r$ in $F(\Phi_{g_0}(\mathcal{W}), M_2(\mathcal{W}))^{\beta}$ such that $r^*br=a$. 
\end{pro}
\begin{proof}
We obtain the conclusion by essentially the same proof as in \cite[Lemma 2.6]{Na6}. 
We shall give a proof for the reader's convenience. 
By \cite[Proposition 3.11]{Na5}, it suffices to show that 
$\mathcal{M}_{\omega}(\Phi_{g_0} (\mathcal{W}), M_2(\mathcal{W}))^{\tilde{\beta}}$ is a factor. 
Since the characteristic invariant of $\tilde{\gamma}|_{C(g_0)}$ is trivial, there exists 
a unitary representation $V$ of $N(\tilde{\gamma}|_{C(g_0)})$ on 
$\pi_{\tau_{\mathcal{W}}}(\mathcal{W})^{''}$ such that $\tilde{\gamma}_h=\mathrm{Ad}(V_h)$ and 
$\tilde{\gamma}_g(V_h)=V_{ghg^{-1}}$ for any $g\in C(g_0)$ and 
$h\in N(\tilde{\gamma}|_{C(g_0)})$. In particular, we have $\tilde{\gamma}_{g_0}(V_h)=V_{h}$ for any 
$h\in N(\tilde{\gamma}|_{C(g_0)})$. This implies that we have 
$$
\left(\begin{array}{cc}
                         V_h   &     0    \\ 
                         0   &       V_h          
 \end{array} \right)
\in \varrho (\Phi_{g_0}(\mathcal{W}))
$$
for any $h\in N(\tilde{\gamma}|_{C(g_0)})$. 
Hence $\tilde{\beta}_h$ is the trivial automorphism of 
$\mathcal{M}_{\omega}(\Phi_{g_0} (\mathcal{W}), M_2(\mathcal{W}))$ for any 
$h\in N(\tilde{\gamma}|_{C(g_0)})$ because we have 
\begin{align*}
\tilde{\beta}_h(\left(\left(\begin{array}{cc}
                         T_{1,n}   &     T_{2,n}    \\ 
                         T_{3,n}   &     T_{4,n}          
 \end{array} \right)\right)_n) 
&= 
 \left(\left(\begin{array}{cc}
                         \tilde{\gamma}_{h}(T_{1,n})   &     \tilde{\gamma}_{h}(T_{2,n})    \\ 
                         \tilde{\gamma}_{h}(T_{3,n})   &     \tilde{\gamma}_{h}(T_{4,n})          
 \end{array} \right)\right)_n \\
&=  \left(\left(\begin{array}{cc}
                         V_{h}T_{1,n}V_h^*   &     V_{h}T_{2,n}V_h^*    \\ 
                         V_{h}T_{3,n}V_h^*   &     V_{h}T_{4,n}V_h^*          
 \end{array} \right)\right)_n
\\
&= \left(\begin{array}{cc}
                         V_h   &     0    \\ 
                         0   &       V_h          
 \end{array} \right)
\left(\left(\begin{array}{cc}
                         T_{1,n}   &     T_{2,n}    \\ 
                         T_{3,n}   &     T_{4,n}          
 \end{array} \right)\right)_n
\left(\begin{array}{cc}
                         V_h   &     0    \\ 
                         0   &       V_h          
 \end{array} \right)^*
\\
&= \left(\left(\begin{array}{cc}
                         T_{1,n}   &     T_{2,n}    \\ 
                         T_{3,n}   &     T_{4,n}          
 \end{array} \right)\right)_n
\end{align*}
for any $\left(\left(\begin{array}{cc}
                         T_{1,n}   &     T_{2,n}    \\ 
                         T_{3,n}   &     T_{4,n}          
 \end{array} \right)\right)_n\in \mathcal{M}_{\omega}(\Phi_{g_0} (\mathcal{W}), M_2(\mathcal{W}))$. 
Therefore  $\gamma|_{C(g_0)}$ induces 
an action $\delta$ of $C(g_0)/ N(\tilde{\gamma}|_{C(g_0)})$ on 
$\mathcal{M}_{\omega}(\Phi_{g_0} (\mathcal{W}), M_2(\mathcal{W}))$ such that 
$$
\mathcal{M}_{\omega}(\Phi_{g_0} (\mathcal{W}), M_2(\mathcal{W}))^{\tilde{\beta}}
=\mathcal{M}_{\omega}(\Phi_{g_0} (\mathcal{W}), M_2(\mathcal{W}))^{\delta}.
$$
By \cite[Theorem 3.2]{C3},  $\tilde{\gamma}_g\otimes\mathrm{id}_{M_2(\mathbb{C})}$ is 
a non-trivial automorphism of $\mathcal{M}_{\omega}(M_{2}(\mathcal{W}))$ for any 
$g\in C(g_0)\setminus N(\tilde{\gamma}|_{C(g_0)})$. 
Hence \cite[Lemma 5.6]{Oc} implies that  $\gamma$ induces an action of 
$C(g_0)/ N(\tilde{\gamma}|_{C(g_0)})$ on $\mathcal{M}_{\omega}(M_{2}(\mathcal{W}))$ that 
satisfies the assumption of \cite[Theorem 6.1]{Oc}.  
Therefore we see that $\mathcal{M}_{\omega}(\Phi_{g_0} (\mathcal{W}), M_2(\mathcal{W}))^{\delta}$ 
is a factor 
by \cite[Proposition 3.14]{Na5} and \cite[Remark 3.15]{Na5}. 
Consequently, we obtain the conclusion. 
\end{proof}

Using Theorem \ref{thm:property-W} and Proposition \ref{pro:strict-comparison-central} instead of 
\cite[Lemma 2.5]{Na6} and \cite[Lemma 2.6]{Na6}, we obtain the following lemma by 
essentially the same argument as in the proof of \cite[Theorem 2.7]{Na6} 
(see also the proof of \cite[Lemma 6.2]{Na3}). We shall give a proof for the reader's convenience. 

\begin{lem}
Let $\gamma$ be an outer action of a finite group $\Gamma$ on $\mathcal{W}$. 
Assume that the characteristic invariant of $\tilde{\gamma}$ is trivial. 
If $g_0$ is an element in $\Gamma$ such that $\gamma|_{C(g_0)}$ has property W, 
then there exists a unitary element $[(v_n)_n]$ in 
$F(\mathcal{W}^{\gamma|_{C(g_0)}})$ such that $\gamma_{g_0}(x)= (v_nxv_n^*)_n$ 
in $\mathcal{W}^{\omega}$ for any $x\in \mathcal{W}$. 
\end{lem}
\begin{proof}
For any $\varepsilon>0$, there exists a projection $[(q_{\varepsilon, n})_n]$ in 
$F(\mathcal{W})^{\gamma|_{C(g_0)}}$ such that 
$\tau_{\mathcal{W}, \omega}([(q_{\varepsilon, n})_n])=1-\varepsilon$ 
by Theorem \ref{thm:property-W} (2). 
Put 
$$
P:= \left[ \left(\left(\begin{array}{cc}
                         h_{n}   &     0    \\ 
                         0       &     0          
 \end{array} \right)\right)_n\right],\quad 
Q:=\left[ \left(\left(\begin{array}{cc}
                         0   &     0    \\ 
                         0   &     h_n          
 \end{array} \right)\right)_n\right]
$$
and
$$
Q_{\varepsilon}:=\left[ \left(\left(\begin{array}{cc}
                         0   &     0    \\ 
                         0   &     q_{\varepsilon, n}         
 \end{array} \right)\right)_n\right]
$$
in $F(\Phi_{g_0}(\mathcal{W}), M_2(\mathcal{W}))^{\beta}$
where $\{h_n\}_{n\in\mathbb{N}}$ is an approximate unit for $\mathcal{W}^{\gamma|_{C(g_0)}}$. 
Proposition \ref{pro:strict-comparison-central} implies that there exists an element $R_{\varepsilon}$ 
in $F(\Phi_{g_0}(\mathcal{W}), M_2(\mathcal{W}))^{\beta}$ such that 
$R_{\varepsilon}^*PR_{\varepsilon}=Q_{\varepsilon}$. By a diagonal argument, we see that there exist
an element $R$ in  $F(\Phi_{g_0}(\mathcal{W}), M_2(\mathcal{W}))^{\beta}$ and a projection 
$[(q_n)_n]$ in $F(\mathcal{W})^{\gamma|_{C(g_0)}}$ such that 
$$
R^* P R= \left[ \left(\left(\begin{array}{cc}
                         0   &     0    \\ 
                         0   &     q_{n}         
 \end{array} \right)\right)_n\right]
$$
and $\tau_{\mathcal{W}, \omega}([(q_{n})_n])=1$. Note that there exists a projection 
$[(p_n)_n]\in F(\mathcal{W})^{\gamma|_{C(g_0)}}$ such that 
$$
PR\left[ \left(\left(\begin{array}{cc}
                         0   &     0    \\ 
                         0   &     q_{n}         
 \end{array} \right)\right)_n\right]R^*P= \left[ \left(\left(\begin{array}{cc}
                         p_{n}   &     0    \\ 
                         0       &     0          
 \end{array} \right)\right)_n\right]
$$
by the definition of $P$. Furthermore, we have $\tau_{\mathcal{W}, \omega}([(p_{n})_n])=1$. 
By Theorem \ref{thm:property-W} (3), there exist 
elements $[(s_n)_n]$ and $[(t_n)_n]$ in $F(\mathcal{W})^{\gamma|_{C(g_0)}}$ such that 
$[(s_n^*s_n)_n]=[(t_n^*t_n)_n]=1=[(h_n)_n]$, $[(s_ns_n^*)_n]=[(q_{n})_n]$ and 
$[(t_nt_n^*)_n]=[(p_n)_n]$. Put  
$$
V=
\left[ \left(\left(\begin{array}{cc}
                         0   &     0    \\ 
                         0   &     s_n^*          
 \end{array} \right)\right)_n\right]
R^*\left[ \left(\left(\begin{array}{cc}
                         t_{n}   &     0    \\ 
                         0       &     0          
 \end{array} \right)\right)_n\right]\in F(\Phi_{g_0}(\mathcal{W}), M_2(\mathcal{W}))^{\beta}.
$$
Then we have $P=V^*V$ and $Q=VV^*$. Furthermore, there exists an element $(v_n)_n$ in 
$\mathcal{W}^{\omega}$ such that 
$$
V=\left[ \left(\left(\begin{array}{cc}
                         0   &     0    \\ 
                         v_n   &     0          
 \end{array} \right)\right)_n\right]
$$
and $[(v_n^*v_n)_n]=[(v_nv_n^*)_n]=1$ in $F(\mathcal{W})$ 
by the definition of $P$ and $Q$. Since $V$ is an element of the fixed point algebra by $\beta$, 
we have 
$$
\left[ \left(\left(\begin{array}{cc}
                         0   &     0    \\ 
                         v_n   &     0          
 \end{array} \right)\right)_n\right]=\left[ \left(\left(\begin{array}{cc}
                         0   &     0    \\ 
                         \gamma_g(v_n)   &     0          
 \end{array} \right)\right)_n\right]
$$
for any $g\in C(g_0)$. Replacing $v_n$ with $\frac{1}{|C(g_0)|}\sum_{g\in C(g_0)}\gamma_g(v_n)$, 
we may assume that $(v_n)_n$ is an element in $(\mathcal{W}^{\gamma|_{C(g_0)}})^{\omega}$. 
Furthermore, we have 
$$
(v_nx)_n= (\gamma_{g_0}(x)v_n)_n
$$
in $\mathcal{W}^{\omega}$ for any $x\in \mathcal{W}$ because 
$$
 \left(\left(\begin{array}{cc}
                         0   &     0    \\ 
                         v_n   &     0          
 \end{array} \right)\right)_n\in 
M_2(\mathcal{W})^{\omega}\cap \Phi_{g_0}(\mathcal{W})^{\prime}. 
$$
Note that we have $(v_nxv_n^*)_n=(\gamma_{g_0}(x)v_nv_n^*)_n=\gamma_{g_0}(x)$ and 
$(yv_n)_n =(\gamma_{g_0}(y)v_n)_n=(v_ny)_n$ for any $x\in \mathcal{W}$ and 
$y\in \mathcal{W}^{\gamma|_{C(g_0)}}$. Consequently, we see that $[(v_n)_n]$ is 
the desired unitary element in $F(\mathcal{W}^{\gamma|_{C(g_0)}})$. 
\end{proof}

The following lemma is a corollary of the lemma above. 

\begin{lem}\label{lem:strongly-app-inner}
Let $\gamma$ be an outer action of a finite group $\Gamma$ on $\mathcal{W}$. 
Assume that the characteristic invariant of $\tilde{\gamma}$ is trivial and 
$\gamma|_{C(g)}$ has property W for any $g\in \Gamma$. Then there exists a map 
$w$ from $\Gamma$ to $\mathcal{W}^{\omega}$ such that 
$$
\gamma_g(x)=w_g xw_g^*, \quad w_g^*w_gx=w_gw_g^*x=x \quad 
\text{and} \quad 
\gamma_g(w_h)= w_{ghg^{-1}}
$$
in $\mathcal{W}^{\omega}$ for any $x\in \mathcal{W}$ and $g,h\in \Gamma$. 
\end{lem}
\begin{proof}
Take representatives $\iota$, $g_1$,...,$g_N$ of all the distinct conjugacy classes of $\Gamma$. 
For any $1\leq j \leq N$, there exists a unitary element $[v_j]\in 
F(\mathcal{W}^{\gamma|_{C(g_j)}})$ such that $\gamma_{g_j}(x)= v_jxv_j^*$ in 
$\mathcal{W}^{\omega}$ for any $x\in \mathcal{W}$ by the lemma above. 
Define a map $w$ from $\Gamma$ to $\mathcal{W}^{\omega}$ by  
$w_{\iota}:=(h_n)_n$ where $\{h_n\}_{n=1}^{\infty}$ is an approximate unit for $\mathcal{W}^{\gamma}$ 
and 
$$
w_{gg_{j}g^{-1}}:= \gamma_g(v_j)
$$
for any $g\in \Gamma$ and $1\leq j \leq N$. 
Since we have $\gamma_g(v_j)=v_j$ for any $g\in C(g_j)$, 
it can be easily checked that $w$ is well defined. 
Easy computations show $\gamma_g(x)=w_g xw_g^*$ and $\gamma_g(w_h)=w_{ghg^{-1}}$ 
for any $x\in \mathcal{W}$ and $g, h\in \Gamma$.  
\cite[Proposition 2.3]{Na6} implies $ w_g^*w_gx=w_gw_g^*x=x$ for 
any $x\in \mathcal{W}$ and $g\in \Gamma$. 
\end{proof}

The following lemma is an application of Jones' classification \cite[Corollary 6.2.5]{Jones} 
(or \cite[Theorem 2.6]{Oc}).  

\begin{lem}\label{lem:app-rep-von}
Let $\gamma$ be an outer action of a finite group $\Gamma$ on $\mathcal{W}$. 
If the characteristic invariant of $\tilde{\gamma}$ is trivial, then there exists a unitary representation 
$U$ of $\Gamma$ on $\mathcal{M}^{\omega}(\mathcal{W})$ such that 
the restriction $U|_{N(\tilde{\gamma})}$ is a unitary representation on 
$\pi_{\tau_{\mathcal{W}}}(\mathcal{W})^{''}$, 
$$
\tilde{\gamma}_g(T)=U_gTU_g^* \quad \text{and} \quad 
\tilde{\gamma}_g(U_h)= U_{ghg^{-1}} \quad \text{in} \quad \mathcal{M}^{\omega}(\mathcal{W})
$$ 
for any $T\in \pi_{\tau_{\mathcal{W}}}(\mathcal{W})^{''}$ and $g,h\in \Gamma$. 
\end{lem}
\begin{proof}
Put $G:=\Gamma /N(\tilde{\gamma})$ and $B:=M_{|G|^{\infty}}$. 
Let $\lambda$ be a unitary representation of $\Gamma$ on 
$M_{|G|}(\mathbb{C})$ such that 
$\lambda_g =\sum_{[h]\in G}e_{[gh], [h]}^{G}$ for any $g\in \Gamma$. 
(Note that we have $\mathrm{Ad}(\lambda_g) (e_{[h], [k]})=e_{[gh], [gk]}$ for any $g\in \Gamma$ and 
$[h], [k]\in G$. Also, if $N(\tilde{\gamma})=\{\iota\}$, then $\lambda$ is nothing but 
the left regular representation of $\Gamma$.) 
Define an action $\beta$ of $\Gamma$ on $B\cong \bigotimes_{n=1}^{\infty}
M_{|G|}(\mathbb{C})$ by 
$$
\beta_g:= \bigotimes_{n=1}^{\infty} \mathrm{Ad}(\lambda_g)
$$
for any $g\in \Gamma$. 
Then $\tilde{\beta}$ is an action of $\Gamma$ on the injective II$_1$ factor 
$\pi_{\tau_{B}}(B)^{''}$
such that 
$N(\tilde{\beta})=N(\tilde{\gamma})$. 
Since we have $\lambda_g=1_{M_{|G|}(\mathbb{C})}$ for any $g\in N(\tilde{\gamma})$, 
the characteristic invariant of $\tilde{\beta}$ is trivial. 
Hence $\tilde{\gamma}$ is cocycle conjugate to $\tilde{\beta}$ by \cite[Corollary 6.2.5]{Jones}
(or \cite[Theorem 2.6]{Oc}). 
Therefore there exist an isomorphism $\Theta$ from $\pi_{\tau_{\mathcal{W}}}(\mathcal{W})^{''}$ 
onto $\pi_{\tau_{B}}(B)^{''}$ and a $\tilde{\beta}$-cocycle $V$ 
such that $\Theta \circ \tilde{\gamma}_g=\mathrm{Ad}(V_g)\circ \tilde{\beta}_g \circ \Theta$ for any 
$g\in \Gamma$. 
For any $g\in \Gamma$, put
$$
U_g:= (\Theta^{-1} (V_g)\Theta^{-1}
(\overbrace{\lambda_g \otimes \cdots \otimes \lambda_g}^{n}\otimes 
1_{M_{|G|}(\mathbb{C})} \otimes 1_{M_{|G|}(\mathbb{C})}\otimes \cdots))_n\in \mathcal{M}^{\omega}(\mathcal{W}).
$$ 
It can be easily checked that $U$ is the desired unitary representation. 
\end{proof}

\begin{lem}\label{lem:non-rep}
Let $\gamma$ be an outer action of a finite group $\Gamma$ on $\mathcal{W}$. 
Assume that the characteristic invariant of $\tilde{\gamma}$ is trivial and 
$\gamma|_{C(g)}$ has property W for any $g\in \Gamma$. Then there exist a unitary representation 
$V$ of $N(\tilde{\gamma})$ on $\pi_{\tau_{\mathcal{W}}}(\mathcal{W})^{''}$ and a map $v$ from 
$\Gamma$ to $\mathcal{W}^{\omega}$ such that  $\varrho \circ v$ is a unitary representation of 
$\Gamma$ on $\mathcal{M}^{\omega}(\mathcal{W})$, 
$$
\gamma_g (x)= v_gxv_g^*,  \quad 
\gamma_g(v_h)=v_{ghg^{-1}}, \quad v_g^*v_gx=v_gv_g^*x=x
$$
in $\mathcal{W}^{\omega}$ and 
$$
\varrho (v_k)=V_k
$$
for any $x\in \mathcal{W}$, $g,h\in \Gamma$ and $k\in N(\tilde{\gamma})$. 
\end{lem}
\begin{proof}
Take representatives $\iota$, $g_1$,...,$g_N$ of all the distinct conjugacy classes of $\Gamma$, 
a map $w$ from $\Gamma$ to $\mathcal{W}^{\omega}$ as in Lemma \ref{lem:strongly-app-inner}  
and a unitary representation $U$ of $\Gamma$ on $\mathcal{M}^{\omega}(\mathcal{W})$ as in 
Lemma \ref{lem:app-rep-von}. 
By properties of $w$ and $U$, we have $\varrho (w_g^*) U_g\pi_{\tau_{\mathcal{W}}}(x)=
\pi_{\tau_{\mathcal{W}}}(x)\varrho (w_g^*)U_g$ and 
$\tilde{\gamma}_h(\varrho (w_g^*)U_g)= \varrho (w_{hgh^{-1}}^*)U_{hgh^{-1}}
=\varrho (w_g^*)U_g$ for any $x\in \mathcal{W}$, $g\in \Gamma$ and $h\in C(g)$. Hence 
$\varrho (w_g^*)U_g$ is a unitary element in 
$\mathcal{M}_{\omega}(\mathcal{W})^{\tilde{\gamma}|_{C(g)}}$. 
Since $\mathcal{M}_{\omega}(\mathcal{W})^{\tilde{\gamma}|_{C(g)}}$ is a von Neumann algebra, 
there exists a self-adjoint element $T_g$ in $\mathcal{M}_{\omega}(\mathcal{W})^{\tilde{\gamma}|_{C(g)}}$ 
such that $\varrho (w_g^*)U_g=e^{iT_g}$. The surjectivity of $\varrho$ implies that there exists 
a self-adjoint element $t_g$ in $F(\mathcal{W})^{\gamma|_{C(g)}}$ such that 
$\varrho (t_g)=T_g$. 
Therefore, for any $1\leq j \leq N$, there exists a unitary element $[z_{j}]=e^{it_{g_j}}$ in 
$F(\mathcal{W})^{\gamma|_{C(g_j)}}$ such that $\varrho([z_{j}])=\varrho (w_{g_j}^*)U_{g_j}$. 
Replacing $z_{j}$ with $\frac{1}{|C(g_j)|}\sum_{g\in C(g_j)} \gamma_g(z_{j})$, 
we may assume that $\gamma_g(z_{j})=z_{j}$ for any $g\in C(g_j)$. 
Define a map $z$ from $\Gamma$ to $\mathcal{W}^{\omega}\cap \mathcal{W}^{\prime}$ by 
$z_{\iota}:= (h_n)_n$ where $\{h_n\}_{n\in\mathbb{N}}$ is an approximate unit for $\mathcal{W}^{\gamma}$ 
and 
$$
z_{gg_{j}g^{-1}}:= \gamma_g(z_{j}) 
$$
for any $g\in \Gamma$ and $1\leq j \leq N$. 
Similar arguments as in the proof of Lemma \ref{lem:strongly-app-inner} show that $z$ is a well 
defined map such that $\gamma_g(z_h)=z_{ghg^{-1}}$ and $z_g^*z_gx=z_gz_g^*x=x$
for any $x\in \mathcal{W}$ and $g, h\in\Gamma$.  Put $v_g:= w_gz_g$ for any $g\in \Gamma$ 
and $V:=U|_{N(\tilde{\gamma})}$. Then it can be easily checked that $v$ and $V$ have the desired 
property. 
\end{proof}

The following theorem is the main theorem in this section. 

\begin{thm}\label{thm:app-rep}
Let $\gamma$ be an outer action of a finite group $\Gamma$ on $\mathcal{W}$. 
Assume that the characteristic invariant of $\tilde{\gamma}$ is trivial and 
$\gamma|_{C(g)}$ has property W for any $g\in \Gamma$. Then there exist a unitary representation 
$V$ of $N(\tilde{\gamma})$ on $\pi_{\tau_{\mathcal{W}}}(\mathcal{W})^{''}$ and 
maps $u$ and $w$ from $\Gamma$ to $\mathcal{W}^{\omega}$ such that 
$$
\gamma_g (x)= u_gxu_g^*,  \quad w_gx=xw_g, \quad \gamma_g(u_h)x=w_g^*u_{ghg^{-1}}w_gx,  \quad 
\gamma_g(w_h)= w_h, 
$$
$$
u_{gh}x=u_gu_hx,  \quad w_{gh}x=w_{g}w_{h}x, \quad 
u_{g}^*u_gx=u_{g}u_{g}^*x=x, \quad  w_{g}^*w_gx=w_{g}w_{g}^*x=x
$$
in $\mathcal{W}^{\omega}$ and 
$$
\varrho (u_k)=V_k, \quad \varrho (w_g)= 1_{\mathcal{M}_{\omega}(\mathcal{W})}
$$
in $\mathcal{M}^{\omega}(\mathcal{W})$ 
 for any $x\in \mathcal{W}$, $g,h\in \Gamma$ and $k\in N(\tilde{\gamma})$.
\end{thm}
\begin{proof}
Note that $\gamma=\gamma|_{C(\iota)}$ has property W. Let $\{h_n\}_{n\in\mathbb{N}}$ be an 
approximate unit for $\mathcal{W}^{\gamma}$. 
Take a unitary representation $V$ of $N(\tilde{\gamma})$ on 
$\pi_{\tau_{\mathcal{W}}}(\mathcal{W})^{''}$ and a map $v$ from 
$\Gamma$ to $\mathcal{W}^{\omega}$ as in Lemma \ref{lem:non-rep}. 
Since we have $v_gv_hv_{gh}^*x=xv_{g}v_hv_{gh}^*$ and $\varrho (v_{g}v_hv_{gh}^*)
=1_{\mathcal{M}_{\omega}(\mathcal{W})}$ 
for any $x\in \mathcal{W}$ and $g,h\in \Gamma$, 
$F:=\{1_{F(\mathcal{W})}-[v_gv_hv_{gh}^*]\; |\; g,h\in \Gamma \}$ 
is a finite subset of $\mathrm{ker}\; \varrho|_{F(\mathcal{W})}$. 
Applying Proposition \ref{pro:small-projection} to $F$, 
we obtain a projection $p=[(p_n)_n]$ in $\mathrm{ker}\;\varrho|_{F(\mathcal{W})^{\gamma}}$ 
such that $p(1_{F(\mathcal{W})}-[v_gv_hv_{gh}^*])
=(1_{F(\mathcal{W})}-[v_gv_hv_{gh}^*])p=1_{F(\mathcal{W})}-[v_gv_hv_{gh}^*]$ for any $g,h \in \Gamma$. 
Replacing $(p_n)_n$ with $(\frac{1}{|\Gamma|}\sum_{g\in \Gamma}\gamma_g(p_n))_n$, 
we may assume that 
$(p_n)_n$ is a positive contraction in $(\mathcal{W}^{\gamma})^{\omega}\cap \mathcal{W}^{\prime}$. 
Note that we have 
$$
(h_n-p_n)_n^2x=(h_n-p_n)_nx=(h_n-p_n)_nv_{g}v_{h}v_{gh}^*x=v_{g}v_{h}v_{gh}^*(h_n-p_n)_n x
$$ 
in $\mathcal{W}^{\omega}$ for any $x\in \mathcal{W}$ and $g,h\in\Gamma$.  

Since $\gamma$ has property W, we see that there 
exists a system of matrix units $\{e_{g, h}^{\Gamma}\}_{g, h\in \Gamma}$ in 
$F(\mathcal{W})^{\gamma}$ such that 
$\sum_{g\in\Gamma} e_{g,g}^{\Gamma} = 1_{F(\mathcal{W})}$
by Definition 1.6(i). Let 
$\{E_{g, h}^{\Gamma}\}_{g, h\in \Gamma}$ be representatives of 
$\{e_{g, h}^{\Gamma}\}_{g, h\in \Gamma}$. We may assume that 
$\{E_{g, h}^{\Gamma}\}_{g, h\in \Gamma}$ are in $\mathcal{W}^{\omega}\cap
\mathcal{W}^{\prime}\cap \{v_g\; :\; g\in \Gamma\}^{\prime}\cap
\{(p_n)_n\}^{\prime}$ by taking suitable subsequences (see, for example, 
[28, Lemma 5.3] for such a technique). Furthermore, we may assume that 
$\{E_{g, h}^{\Gamma}\}_{g, h\in\Gamma}$ are in $(\mathcal{W}^{\gamma})^{\omega}$ by 
replacing suitable sequences as in $(p_n)_n$.

For any $g\in \Gamma$, put
$$
z_g:= (h_n- p_n)_n +
\sum_{h\in \Gamma}v_gv_hv_{gh}^*E_{h, gh}^{\Gamma}(p_n)_n 
\in\mathcal{W}^{\omega}\cap \mathcal{W}^{\prime}
$$
and 
$$
u_g:= z_g^*v_g\in \mathcal{W}^{\omega}.
$$
Since we have $\varrho (z_g)= \varrho((h_n)_n)=1$ for any $g\in\Gamma$, 
$$
\varrho (u_k)=V_k 
$$
for any $k\in N(\tilde{\gamma})$. 
Similar computations as in the proof of \cite[Lemma 2.4]{Na6} show 
\begin{align*}
\left(\sum_{k\in \Gamma}v_gv_kv_{gk}^*E_{k, gk}^{\Gamma}\right)^* &
\left(\sum_{k\in \Gamma}v_gv_kv_{gk}^*E_{k, gk}^{\Gamma}\right) x \\
&=\sum_{k\in \Gamma}E_{k,k}^{\Gamma}x
=\left(\sum_{k\in \Gamma}v_gv_kv_{gk}^*E_{k, gk}^{\Gamma}\right)
\left(\sum_{k\in \Gamma}v_gv_kv_{gk}^*E_{k, gk}^{\Gamma}\right)^*x 
\end{align*}
and
\begin{align*}
& v_g\left(\sum_{k\in \Gamma}v_hv_kv_{hk}^*E_{k, hk}^{\Gamma}\right)v_g^*
\left(\sum_{l\in \Gamma}v_gv_lv_{gl}^*E_{l, gl}^{\Gamma}\right)
\left(\sum_{k\in \Gamma}v_{gh}v_{k}v_{ghk}^*E_{k, ghk}^{\Gamma}\right)^*
\sum_{k\in\Gamma} E_{k,k}^{\Gamma}x \\
& =v_gv_hv_{gh}^*\sum_{k\in\Gamma} E_{k,k}^{\Gamma}x
\end{align*}
for any $x\in \mathcal{W}$ and $g,h\in \Gamma$. 
Hence it can be easily checked that we have 
$$
z_g^*z_gx= z_gz_g^*x=x\quad \text{and} \quad v_gz_hv_g^*z_{g}z_{gh}^* x=v_{g}v_{h}v_{gh}^*x
$$
for any $x\in \mathcal{W}$ and $g,h\in \Gamma$. Since we have 
$v_gx=\gamma_g(x)v_g$ and $z_gx=xz_g$ for any $x\in \mathcal{W}$ and $g\in \Gamma$, 
these equalities imply 
$$
u_g^*u_g x=u_gu_g^* x=x, \quad u_g xu_g^*= \gamma_g(x).
$$
and 
$$
z_{gh}^*v_{gh} \gamma_{h^{-1}g^{-1}}(x)=z_g^*v_gz_h^*v_h\gamma_{h^{-1}g^{-1}}(x)
$$
for any $x\in\mathcal{W}$ and $g,h\in \Gamma$. 
Note that we have 
$$
u_gu_hx=u_{gh}x  
$$
for any $x\in\mathcal{W}$ and $g,h\in \Gamma$ by the last equation. 
For any $g\in \Gamma$, put 
$$
w_g:= (h_n- p_n)_n +\sum_{h\in \Gamma}E_{ghg^{-1}, h}^{\Gamma}(p_n)_n 
\in  (\mathcal{W}^{\gamma})^{\omega}\cap \mathcal{W}^{\prime}.
$$
Then we have 
$$
w_{gh}x=w_{g}w_{h}x, \quad w_g^*w_gx=w_gw_g^*x=x \quad \text{and} 
\quad \varrho(w_g)=1_{\mathcal{M}_{\omega}(\mathcal{W})} \
$$
for any $g\in\Gamma$. Note that we have 
\begin{align*}
&\left(\sum_{k\in \Gamma} E_{gkg^{-1}, k}^{\Gamma}\right)^*
\left(\sum_{k\in \Gamma}v_{ghg^{-1}}v_{k}v_{ghg^{-1}k}^*E_{k, ghg^{-1}k}^{\Gamma}\right)^*
v_{ghg^{-1}}\left(\sum_{k\in \Gamma} E_{gkg^{-1}, k}^{\Gamma}\right) x \\
&=\left(\sum_{k\in \Gamma} E_{k, gkg^{-1}}^{\Gamma}\right)
\left(\sum_{k\in \Gamma}\sum_{k^{\prime}\in \Gamma }v_{ghg^{-1}k}v_{k}^*
v_{ghg^{-1}}^*E_{ghg^{-1}k, k}^{\Gamma}E_{gk^{\prime}g^{-1}, k^{\prime}}
\gamma_{ghg^{-1}}(x) v_{ghg^{-1}} \right)\\
&= \left(\sum_{k\in \Gamma} E_{k, gkg^{-1}}^{\Gamma}\right)
\left(\sum_{k^{\prime}\in \Gamma }v_{ghk^{\prime}g^{-1}}v_{gk^{\prime}g^{-1}}^*
v_{ghg^{-1}}^*E_{ghk^{\prime}g^{-1}, k^{\prime}}^{\Gamma}
\gamma_{ghg^{-1}}(x)v_{ghg^{-1}} \right) \\
&= \sum_{k\in \Gamma} 
\sum_{k^{\prime}\in \Gamma }\gamma_g(v_{hk^{\prime}})\gamma_g(v_{k^{\prime}}^*)
\gamma_g(v_{h}^*)E_{k, gkg^{-1}}^{\Gamma}E_{ghk^{\prime}g^{-1}, k^{\prime}}^{\Gamma}
\gamma_{ghg^{-1}}(x) \gamma_g(v_{h}) \\
&= \sum_{k^{\prime}\in \Gamma }\gamma_g(v_{hk^{\prime}})\gamma_g(v_{k^{\prime}}^*)
\gamma_g(v_{h}^*)E_{hk^{\prime}, k^{\prime}}^{\Gamma}\gamma_{ghg^{-1}}(x) \gamma_g(v_{h}) \\
&= \gamma_g\left(\left(\sum_{k^{\prime}\in \Gamma}v_{h}v_{k^{\prime}}v_{hk^{\prime}}^* 
E_{k^{\prime}, hk^{\prime}}^{\Gamma}\right)^* v_{h}\right) x
\end{align*}
for any $x\in \mathcal{W}$ and $g, h\in\Gamma$.  
Hence it can be easily checked that we have  
$$
w_g^*u_{ghg^{-1}}w_gx=w_g^* z_{ghg^{-1}}^{*}v_{ghg^{-1}}w_gx= \gamma_g(z_{h}^*v_{h})x
=\gamma_g(u_h)x 
$$
for any $x\in \mathcal{W}$ and $g, h\in\Gamma$.
Consequently, $u$ and $w$ are the desired maps. 
\end{proof}

\section{Absorption of actions with trivial characteristic invariant} \label{sec:abs}

In this section we shall show an absorption result for certain finite group actions with trivial 
characteristic invariant (in $\pi_{\tau_{\mathcal{W}}}(\mathcal{W})^{''}$)  by using 
Szab\'o's approximate cocycle intertwining argument in \cite{Sza7} (see also \cite{Ell2} and 
\cite[Section 5]{Na5}). 
Note that this result is a key ingredient in the proof of the main classification theorems in the 
next section. 

Let $\alpha$ be an outer action of a countable discrete amenable group 
$\Gamma$ on $\mathcal{W}$, and let $(\Psi, u)$ be a sequential asymptotic cocycle morphism 
from $(\mathcal{W}, \alpha)$ to $(\mathcal{W}, \alpha)$ such that 
$$
\Psi (x)=(x)_n \quad \text{and} \quad \varrho(u_h) =1_{\mathcal{M}^{\omega}(\mathcal{W})}
$$ 
for any $x\in \mathcal{W}$ and $h\in N(\tilde{\alpha})$.
Define a homomorphism $\Phi$ from $\mathcal{W}$ to $M_2(\mathcal{W})^{\omega}$ by 
$$
\Phi (x):= \left(\left(\begin{array}{cc}
                         x   &     0    \\ 
                         0   &     x          
 \end{array} \right)\right)_n
$$
for any $x\in\mathcal{W}$. Since $(\Psi ,u)$  is a sequential asymptotic cocycle morphism 
from $(\mathcal{W}, \alpha)$ to $(\mathcal{W}, \alpha)$ such that 
$\tau_{\mathcal{W}, \omega}\circ \Psi=\tau_{\mathcal{W}}$, we can define an action 
$\alpha^{\omega}$ of $\Gamma$ on $F(\Phi (\mathcal{W}), M_2(\mathcal{W}))$ by 
$$
\alpha_g^{\omega}:= \mathrm{Ad}\left( \left(\begin{array}{cc}
                         1_{M(\mathcal{W})}   &     0    \\ 
                         0   &     u_g          
 \end{array} \right)\right)\circ (\alpha_g\otimes\mathrm{id}_{M_2(\mathbb{C})})
$$
for any $g\in \Gamma$ as in \cite[Section 6]{Na5}.

\begin{pro}
With notation as above, if $a$ and $b$ are positive elements in $F(\Phi (\mathcal{W}), 
M_2(\mathcal{W}))^{\alpha^{\omega}}$ satisfying 
$
d_{\tau_{M_2(\mathcal{W}), \omega}}(a)< 
d_{\tau_{M_2(\mathcal{W}), \omega}}(b)
$,
then there exists an element $r$ in 
$F(\Phi (\mathcal{W}), M_2(\mathcal{W}))^{\alpha^{\omega}}$ such that $r^*br=a$. 
\end{pro}
\begin{proof}
By \cite[Proposition 3.11]{Na5}, it suffices to show that 
$\mathcal{M}_{\omega}(\Phi (\mathcal{W}), M_2(\mathcal{W}))^{\tilde{\alpha}^{\omega}}$ is a factor. 
Since $\varrho (u_h)= 1_{\mathcal{M}^{\omega}(\mathcal{W})}$ for any $h\in  N(\tilde{\alpha})$, we have 
$\tilde{\alpha}_h^{\omega}=\tilde{\alpha}_h\otimes \mathrm{id}_{M_{2}(\mathbb{C})}$ for any 
$h\in N(\tilde{\alpha})$. Let $V$ be a map from $N(\tilde{\alpha})$ to 
$U(\pi_{\tau_{\mathcal{W}}}(\mathcal{W})^{''})$ such that $\tilde{\alpha}_h=\mathrm{Ad}(V_h)$ 
for any $h\in N(\tilde{\alpha})$. Since we have
$$
\left(\begin{array}{cc}
                         V_h   &     0    \\ 
                         0     &     V_h          
 \end{array} \right)\in \varrho (\Phi (\mathcal{W}))
$$
for any $h\in N(\tilde{\alpha})$, $\tilde{\alpha}_h^{\omega}
=\tilde{\alpha}_h\otimes \mathrm{id}_{M_{2}(\mathbb{C})}$ is a trivial automorphism of
$\mathcal{M}_{\omega}(\Phi (\mathcal{W}), M_2(\mathcal{W}))$ for any $h\in N(\tilde{\alpha})$. 
Hence we can define an action $\delta$ of $\Gamma/ N(\tilde{\alpha})$ on 
$\mathcal{M}_{\omega}(\Phi (\mathcal{W}), M_2(\mathcal{W}))$ such that 
$$
\mathcal{M}_{\omega}(\Phi (\mathcal{W}), M_2(\mathcal{W}))^{\tilde{\alpha}^{\omega}}
=\mathcal{M}_{\omega}(\Phi (\mathcal{W}), M_2(\mathcal{W}))^{\delta}.
$$ 
The rest of the proof is the same as the proof of Proposition \ref{pro:strict-comparison-central}. 
\end{proof}

Using the proposition above and Theorem \ref{thm:property-W} instead of \cite[Proposition 6.1]{Na5} 
and properties in the statement of \cite[Theorem 6.3]{Na5}, we obtain the  following uniqueness 
type theorem by the same argument as in the proof of \cite[Theorem 6.3]{Na5}. 

\begin{thm}
Let $\alpha$ be an outer action of a countable discrete amenable group 
$\Gamma$ on $\mathcal{W}$. Assume that $\alpha$ has property W. 
If $(\Psi, u)$ is a sequential asymptotic cocycle morphism from $(\mathcal{W}, \alpha)$ to 
$(\mathcal{W}, \alpha)$ such that 
$$
\Psi (x)=(x)_n \quad \text{and} \quad \varrho (u_h) =1_{\mathcal{M}^{\omega}(\mathcal{W})}
$$ 
for any $x\in \mathcal{W}$ and $h\in N(\tilde{\alpha})$, then $(\Psi, u)$ is inner. 
\end{thm}

The following corollary is an immediate consequence of the theorem above. 

\begin{cor}\label{cor:uniqueness}
Let $\alpha$ be an outer action of a countable discrete amenable group 
$\Gamma$ on $\mathcal{W}$. Assume that $\alpha$ has property W. 
For any finite subsets $F\subset \mathcal{W}$, $\Gamma_0\subset \Gamma$ and 
$\varepsilon>0$, there exist finite subsets 
$F^{\prime}\subset \mathcal{W}$, $\Gamma_0^{\prime}\subset \Gamma$, $N_0 \subset 
N(\tilde{\alpha})$ and $\delta>0$ such that the following holds. 
If $(\varphi, u)$ is a proper 
$(\Gamma_0^{\prime}, F^{\prime}, \delta)$-approximate cocycle morphism from 
$(\mathcal{W}, \alpha)$ to $(\mathcal{W}, \alpha)$ such that 
$$
\| \varphi (x) - x \| < \delta \quad \text{and} \quad  
\| \pi_{\tau_{\mathcal{W}}}(u_h) -1_{\pi_{\tau_{\mathcal{W}}}(\mathcal{W})^{''}} \|_2< \delta  
$$
for any $x\in F^{\prime}$ and $h\in N_0$, then there exists a unitary element $w$ in 
$\mathcal{W}^{\sim}$ such that 
$$
\| \varphi (x) -wxw^* \| <\varepsilon \quad \text{and} \quad 
\|\varphi (x)(u_g-w\alpha_g (w^*)) \| <\varepsilon
$$
for any $x\in F$ and $g\in \Gamma_0$. 
\end{cor}

The following theorem is an existence type theorem. 

\begin{thm}
Let $\alpha$ and $\gamma$ be outer actions of a finite group $\Gamma$ on $\mathcal{W}$. 
Assume that $\gamma|_{C(g)}$ has property W for any $g\in \Gamma$. 
If the characteristic invariant of $\tilde{\gamma}$ is trivial, then there exist a unitary representation 
$U$ of $N(\tilde{\gamma})$ on 
$\pi_{\tau_{\mathcal{W}\otimes\mathcal{W}}}(\mathcal{W}\otimes\mathcal{W})^{''}$, 
a sequential asymptotic cocycle morphism $(\mathrm{id}_{\mathcal{W}\otimes\mathcal{W}}, v)$ from 
$(\mathcal{W}\otimes\mathcal{W}, \alpha\otimes \gamma)$ to $(\mathcal{W}\otimes \mathcal{W},
\alpha\otimes\mathrm{id}_{\mathcal{W}})$ and 
a sequential asymptotic cocycle morphism $(\mathrm{id}_{\mathcal{W}\otimes\mathcal{W}}, v^{\prime})$ 
from 
$(\mathcal{W}\otimes\mathcal{W}, \alpha\otimes \mathrm{id}_{\mathcal{W}})$ to 
$(\mathcal{W}\otimes \mathcal{W},\alpha\otimes\gamma)$ such that 
$\varrho (v_h)=U_h$ and $\varrho(v_h^{\prime})=U_h^*$ for any $h\in N(\tilde{\gamma})$. 
\end{thm}
\begin{proof}
Let $\{h_n\}_{n\in\mathbb{N}}$ be an approximate unit for $\mathcal{W}$. 
Take a unitary representation $V$ of $N(\tilde{\gamma})$ on 
$\pi_{\tau_{\mathcal{W}}}(\mathcal{W})^{''}$ and maps $u$ and $w$ as  in Theorem \ref{thm:app-rep}. 
Define a unitary representation $U$ of $N(\tilde{\gamma})$ on 
$\pi_{\tau_{\mathcal{W}\otimes\mathcal{W}}}(\mathcal{W}\otimes\mathcal{W})^{''}$ by 
$U_h:= 1_{\pi_{\tau_{\mathcal{W}}}(\mathcal{W})^{''}}\otimes V_h$ for any $h\in N(\tilde{\gamma})$. 
For any $g\in \Gamma$, put $z_g:= (h_n\otimes u_{g,n})_n$ and 
$z_g^{\prime}:=(h_n\otimes u_{g,n}^*w_{g,n})_n$ in 
$(\mathcal{W}\otimes\mathcal{W})^{\omega}$ where $(u_{g,n})_n$ and $(w_{g,n})_n$ are representatives 
of $u_g$ and $w_g$, respectively. It is easy to see that we have 
$$
z_{g}^*z_{g}x=x, \quad z_{g}^{\prime *}z_{g}^{\prime}x=x, \eqno{(4.4.1)}
$$
and 
$$
z_gxz_g^*= (\mathrm{id}_{\mathcal{W}}\otimes 
\gamma_g)(x), \quad z_g^{\prime}x z_g^{\prime *}=(\mathrm{id}_{\mathcal{W}}\otimes \gamma_{g^{-1}})(x)
\eqno{(4.4.2)}
$$
for any $x\in\mathcal{W}\otimes \mathcal{W}$ and $g\in \Gamma$. 
Since we have 
$$
z_{g}(\alpha_g\otimes \mathrm{id}_{\mathcal{W}})
(z_{h})(a\otimes b)=(h_n\alpha_g(h_n)a\otimes u_{g,n}u_{h, n}b)_n= (a\otimes u_{gh, n}b)_n= z_{gh}(a\otimes b) 
$$
and 
\begin{align*}
z^{\prime}_{g}(\alpha_g\otimes \gamma_g)(z^{\prime}_h)(a\otimes b)
&=(h_n\alpha_g(h_n) a\otimes u_{g,n}^*w_{g,n}\gamma_g(u_{h,n}^*)w_{h,n}b)_n \\
&
=(a\otimes u_{g,n}^*w_{g,n}w_{g,n}^*u_{ghg^{-1},n}^*w_{g,n}w_{h,n}b)_n \\
&= (a\otimes u_{gh, n}^*w_{gh, n}b)_n= z_{gh}^{\prime}(a\otimes b)
\end{align*}
for any $a, b\in\mathcal{W}$ and $g,h\in \Gamma$, 
$$
z_{gh}x=z_{g}(\alpha_g\otimes \mathrm{id}_{\mathcal{W}})(z_{h})x\quad \text{and} \quad z^{\prime}_{gh}x
= z^{\prime}_g(\alpha_g\otimes\gamma_g)(z^{\prime}_h)x
\eqno{(4.4.3)}
$$ 
for any $x\in \mathcal{W}$ and $g,h\in \Gamma$. 
Since $\mathcal{W}$ has stable rank one, for any $g\in \Gamma$, there exist unitary elements 
$v_g$ and $v^{\prime}_g$ in $(\mathcal{W}^{\sim})^{\omega}$ such that 
$$
v_gx=z_gx \quad \text{and} \quad v^{\prime}_gx= z_g^{\prime}x
$$ 
for any $x\in\mathcal{W}\otimes \mathcal{W}$ by 
(4.4.1) and a routine argument based on the polar decomposition (see, for example, 
an argument in the proof of \cite[Lemma 7.1]{Na5}). 
It is easy to see that we have 
$$
\varrho (v_h)=U_h \quad \text{and} \quad \varrho(v_h^{\prime})=U_h^*
$$ 
for any $h\in N(\tilde{\gamma})$. By (4.4.2), we have 
$$
v_gxv_g^*=z_gxz_g^*=(\mathrm{id}_{\mathcal{W}}\otimes \gamma_{g})(x) \quad \text{and}
\quad  
v^{\prime}_gxv_g^{\prime *}=z_g^{\prime}xz_g^{\prime *}=(\mathrm{id}_{\mathcal{W}}\otimes 
\gamma_{g^{-1}})(x)
$$ 
for any $x\in \mathcal{W}\otimes \mathcal{W}$ and $g\in \Gamma$. 
These formulas imply 
$$
\mathrm{id}_{\mathcal{W}\otimes \mathcal{W}}\circ (\alpha_g\otimes \gamma_g)= 
(\mathrm{id}_{\mathcal{W}}\otimes \gamma_g)\circ (\alpha_g \otimes \mathrm{id}_{\mathcal{W}}) 
=\mathrm{Ad}(v_g)\circ (\alpha_g \otimes \mathrm{id}_{\mathcal{W}})
$$
and 
$$
\mathrm{id}_{\mathcal{W}\otimes \mathcal{W}}\circ (\alpha_g \otimes \mathrm{id}_{\mathcal{W}})
= (\mathrm{id}_{\mathcal{W}}\otimes \gamma_{g^{-1}})\circ (\alpha_g \otimes \gamma_g)
= \mathrm{Ad}(v^{\prime}_g)\circ (\alpha_g \otimes \gamma_g)
$$
for any $g\in \Gamma$. 
Note that we have $v_gx=(\mathrm{id}_{\mathcal{W}}\otimes \gamma_{g})(x)v_{g}$ and 
$v_g^{\prime}x=(\mathrm{id}_{\mathcal{W}}\otimes \gamma_{g^{-1}})(x)v_{g}^{\prime}$ for any 
$x\in \mathcal{W}\otimes \mathcal{W}$ and $g\in \Gamma$. Therefore (4.4.3) implies 
\begin{align*}
xv_{gh}
&= v_{gh}(\mathrm{id}_{\mathcal{W}}\otimes \gamma_{h^{-1}g^{-1}})(x)
=z_{gh}(\mathrm{id}_{\mathcal{W}}\otimes \gamma_{h^{-1}g^{-1}})(x) \\
&=z_{g}(\alpha_g\otimes \mathrm{id}_{\mathcal{W}})(z_{h})
(\mathrm{id}_{\mathcal{W}}\otimes \gamma_{h^{-1}g^{-1}})(x) \\ 
&= z_{g}(\alpha_g\otimes \mathrm{id}_{\mathcal{W}})(v_{h}(\alpha_{g^{-1}}\otimes 
\gamma_{h^{-1}g^{-1}})(x))  \\
&=z_{g}(\mathrm{id}_{\mathcal{W}}\otimes \gamma_{g^{-1}})(x)(\alpha_g\otimes \mathrm{id}_{\mathcal{W}})(v_{h}) \\
&=v_{g}(\mathrm{id}_{\mathcal{W}}\otimes \gamma_{g^{-1}})(x)(\alpha_g\otimes \mathrm{id}_{\mathcal{W}})(v_{h}) \\
&=xv_{g}(\alpha_g\otimes \mathrm{id}_{\mathcal{W}})(v_{h})
\end{align*}
and
\begin{align*}
xv_{gh}^{\prime}
&= v_{gh}^{\prime}(\mathrm{id}_{\mathcal{W}}\otimes \gamma_{gh})(x)
=z_{gh}^{\prime}(\mathrm{id}_{\mathcal{W}}\otimes \gamma_{gh})(x) \\
&=z_{g}^{\prime}(\alpha_g\otimes \gamma_g)(z_{h}^{\prime})
(\mathrm{id}_{\mathcal{W}}\otimes \gamma_{gh})(x) 
= z_{g}^{\prime}(\alpha_g\otimes \gamma_g)(v_{h}^{\prime}(\alpha_{g^{-1}}\otimes 
\gamma_{h})(x))  \\
&=z_{g}^{\prime}(\mathrm{id}_{\mathcal{W}}\otimes \gamma_{g})(x)
(\alpha_g\otimes \gamma_g)(v_{h}^{\prime}) 
=v_{g}^{\prime}(\mathrm{id}_{\mathcal{W}}\otimes \gamma_{g})(x)
(\alpha_g\otimes \gamma_g)(v_{h}^{\prime}) \\
&=xv_{g}^{\prime}(\alpha_g\otimes \gamma_g)(v_{h}^{\prime})
\end{align*}
for any $x\in \mathcal{W}\otimes \mathcal{W}$ and $g,h\in \Gamma$. 
Consequently, 
$(\mathrm{id}_{\mathcal{W}\otimes\mathcal{W}}, v)$ and 
$(\mathrm{id}_{\mathcal{W}\otimes\mathcal{W}}, v^{\prime})$ are the desired 
sequential asymptotic cocycle morphisms. 
\end{proof}

The following corollary is an immediate consequence of the theorem above. 
\begin{cor}\label{cor:existence}
Let $\alpha$ and $\gamma$ be outer actions of a finite group $\Gamma$ on $\mathcal{W}$. 
Assume that $\gamma|_{C(g)}$ has property W for any $g\in \Gamma$. 
If the characteristic invariant of $\tilde{\gamma}$ is trivial, 
then there exists a unitary representation 
$U$ of $N(\tilde{\gamma})$ on 
$\pi_{\tau_{\mathcal{W}\otimes\mathcal{W}}}(\mathcal{W}\otimes\mathcal{W})^{''}$ 
such that the following holds: \ \\
(i) For any finite set $F\subset \mathcal{W}\otimes \mathcal{W}$ and 
$\varepsilon>0$, there exists a map $u$ from 
$\Gamma$ to $U((\mathcal{W}\otimes\mathcal{W})^{\sim})$ such that 
$$
\| (\alpha_g\otimes \gamma_g )(x) - \mathrm{Ad}(u_g)\circ (\alpha_g\otimes \mathrm{id}_{\mathcal{W}})(x) 
\|< \varepsilon, \quad 
\| x(u_{gh}- u_g(\alpha_g\otimes \mathrm{id}_{\mathcal{W}})(u_h))\| < \varepsilon
$$
and 
$$
\| \pi_{\tau_{\mathcal{W}\otimes\mathcal{W}}}(u_k) - U_k\|_2 < \varepsilon
$$
for any $x\in F$, $g,h\in \Gamma$ and $k\in N(\tilde{\gamma})$. 
\ \\
(ii) For any finite set $F\subset \mathcal{W}\otimes\mathcal{W}$ and $\varepsilon>0$, 
there exists a map $v$ from 
$\Gamma$ to $U( (\mathcal{W}\otimes\mathcal{W})^{\sim})$ such that 
$$
\| (\alpha_g\otimes \mathrm{id}_{\mathcal{W}}) (x) - \mathrm{Ad}(v_g)\circ 
(\alpha_g\otimes \gamma_g)(x) 
\|< \varepsilon, \quad 
\| x(v_{gh}- v_g(\alpha_g\otimes\gamma_g)(v_h))\| < \varepsilon 
$$
and 
$$
\| \pi_{\tau_{\mathcal{W}\otimes\mathcal{W}}}(v_k) - U_k^*\|_2 < \varepsilon
$$
for any $x\in F$, $g,h\in \Gamma$ and $k\in N(\tilde{\gamma})$. 
\end{cor}

The following theorem is the main result in this section. 
\begin{thm}\label{thm;absorption}
Let $\alpha$ and $\gamma$ be outer actions of a finite group $\Gamma$ on $\mathcal{W}$. 
Assume that $\alpha\otimes \gamma$ and 
$\gamma|_{C(g)}$ have property W for any $g\in \Gamma$. 
If the characteristic invariant of 
$\tilde{\gamma}$ is trivial and $N(\tilde{\alpha})\subseteq N(\tilde{\gamma})$, 
then $\alpha\otimes \gamma$ is cocycle conjugate to $\alpha\otimes \mathrm{id}_{\mathcal{W}}$.
\end{thm}
\begin{proof}
Note that $\alpha\otimes\mathrm{id}_{\mathcal{W}}$ also has property W by 
Theorem \ref{thm:property-W} and we have $N(\widetilde{\alpha\otimes\gamma}) 
=N(\widetilde{\alpha\otimes\mathrm{id}_{\mathcal{W}}})=N(\tilde{\alpha})\subseteq N(\tilde{\gamma})$. 
By Corollary \ref{cor:uniqueness} (uniqueness), Corollary \ref{cor:existence} (existence) 
and Szab\'o's approximate cocycle
intertwining argument \cite{Sza7} (see also \cite[Proposition 5.2 and the proof of Theorem 8.1]{Na5}), 
we obtain the conclusion. 
Indeed, let $\{x_n\}_{n\in\mathbb{N}}$ be a dense subset of $\mathcal{W}\otimes\mathcal{W}$.  
Take a unitary representation $U$ of $N(\tilde{\gamma})$ on 
$\pi_{\tau_{\mathcal{W}\otimes\mathcal{W}}}(\mathcal{W}\otimes\mathcal{W})^{''}$ as in 
Corollary \ref{cor:existence}. For any $n\in\mathbb{N}$, put $\varepsilon_n:=1/2^n$. 

Step 1. Applying Corollary \ref{cor:uniqueness} to $\alpha\otimes \gamma$,  
$F_1:=\{x_1, x_1^*\}$, $\Gamma$ and $\varepsilon_1/2$, 
we obtain $F^{\prime}_1$ and $\delta_1>0$. (Note that we need not consider  
$\Gamma^{\prime}$ and $N_0$ because $\Gamma$ is a finite group.) We may assume that 
$\delta_1<\varepsilon_1$. Put $F_{1}^{\prime\prime}:= F_1\cup F_1^{\prime}$. 
By Corollary \ref{cor:existence} (i), there exists  a map $u^{(1)}$ from 
$\Gamma$ to $U((\mathcal{W}\otimes\mathcal{W})^{\sim})$ such that 
$$
\| (\alpha_g\otimes \gamma_g )(x) - \mathrm{Ad}(u^{(1)}_g)\circ 
(\alpha_g\otimes \mathrm{id}_{\mathcal{W}})(x) 
\|< \frac{\delta_1}{2},  \eqno{(4.6.1)}
$$
$$
\| x(u_{gh}^{(1)}- u_g^{(1)}(\alpha_g\otimes \mathrm{id}_{\mathcal{W}})(u_h^{(1)}))\| < 
\frac{\delta_1}{2}
\eqno{(4.6.2)}
$$
and 
$$
\| \pi_{\tau_{\mathcal{W}\otimes\mathcal{W}}}(u_k^{(1)}) - U_k\|_2 < \frac{\delta_1}{2} \eqno{(4.6.3)}
$$
for any $x\in F_1^{\prime\prime}$, $g,h\in \Gamma$ and $k\in N(\tilde{\gamma})$. 
Put $(\varphi_1, \tilde{u}^{(1)}):= (\mathrm{id}_{\mathcal{W}\otimes\mathcal{W}}, u^{(1)})$. Note that 
$(\varphi_1, \tilde{u}^{(1)})$ is a proper 
$(\Gamma, F_1, \varepsilon_1)$-approximate cocycle morphism from 
$(\mathcal{W}\otimes\mathcal{W}, \alpha\otimes\gamma)$ to 
$(\mathcal{W}\otimes\mathcal{W}, \alpha\otimes\mathrm{id}_{\mathcal{W}})$. 

Step 2. Choose a finite self-adjoint subset $G_1$ of $\mathcal{W}\otimes\mathcal{W}$ such that 
$$
F_1=\varphi_{1}(F_1) \subset G_1 \quad  \text{and} \quad \{\tilde{u}^{(1)}_g\; |\; g\in\Gamma\}
\subset \{x+ \lambda 1_{(\mathcal{W}\otimes\mathcal{W})^{\sim}}\; |\; x\in G_1, \lambda\in\mathbb{C}\}. 
$$
Applying Corollary \ref{cor:uniqueness} to $\alpha\otimes \mathrm{id}_{\mathcal{W}}$,  
$G_1$, $\Gamma$ and $\varepsilon_1/2$, we obtain $G_1^{\prime}$ and $\delta_2>0$. 
We may assume 
$$
\delta_2 < \dfrac{\delta_1}{1+\displaystyle{\max_{x\in F_1^{\prime}}\| x\|} }. \eqno{(4.6.4)}
$$
Put $G_1^{\prime\prime}:= G_1\cup G_1^{\prime}\cup \bigcup_{g\in \Gamma} 
F_1^{\prime}u_{g}^{(1)}$. 
Note that we have 
$$
x\in G_1^{\prime\prime}, \quad xu_g^{(1)}\in G_1^{\prime\prime}, \quad 
u_g^{(1)}=\tilde{u}^{(1)}_g\in \{x+ \lambda 1_{(\mathcal{W}\otimes\mathcal{W})^{\sim}}\; |\; 
x\in G_1^{\prime\prime}, \lambda\in\mathbb{C}\} \eqno{(4.6.5)}
$$
for any $x\in F_1^{\prime}$ and $g\in \Gamma$. 
By Corollary \ref{cor:existence} (ii), there exists  a map $v^{(1)}$ from 
$\Gamma$ to $U((\mathcal{W}\otimes\mathcal{W})^{\sim})$ such that 
$$
\| (\alpha_g\otimes \mathrm{id}_{\mathcal{W}} )(x) - \mathrm{Ad}(v^{(1)}_g)\circ 
(\alpha_g\otimes \gamma_g)(x) 
\|< \frac{\delta_2}{2},  \eqno{(4.6.6)}
$$
$$
\| x(v_{gh}^{(1)}- v_g^{(1)}(\alpha_g\otimes \gamma_g)(v_h^{(1)}))\| < 
\frac{\delta_2}{2}
\eqno{(4.6.7)}
$$
and 
$$
\| \pi_{\tau_{\mathcal{W}\otimes\mathcal{W}}}(v_k^{(1)}) - U_k^*\|_2 < \frac{\delta_2}{2} \eqno{(4.6.8)}
$$
for any $x\in G_1^{\prime\prime}$, $g,h\in \Gamma$ and $k\in N(\tilde{\gamma})$. 
For any $g\in \Gamma$, 
put $z_g^{(1)}:= u_g^{(1)}v_g^{(1)}\in U((\mathcal{W}\otimes\mathcal{W})^{\sim})$. 
Then we have 
\begin{align*}
 \| x(z_{gh}^{(1)}-z_g^{(1)}(\alpha_g&\otimes \gamma_g)(z_h^{(1)}))\| \\
& = \| xu_{gh}^{(1)}v_{gh}^{(1)}-xu_{g}^{(1)}v_{g}^{(1)}(\alpha_g\otimes \gamma_g)(u_{h}^{(1)}v_{h}^{(1)})\| \\
(4.6.5), (4.6.7)\; & <
\| x u_{gh}^{(1)} v_g^{(1)}(\alpha_g\otimes \gamma_g)(v_h^{(1)})-
xu_{g}^{(1)}v_{g}^{(1)}(\alpha_g\otimes \gamma_g)(u_{h}^{(1)}v_{h}^{(1)})\|+ \frac{\delta_2}{2} \\
(4.6.2)\; & < \| xu_g^{(1)}(\alpha_g\otimes \mathrm{id}_{\mathcal{W}})(u_h^{(1)})
 v_g^{(1)}(\alpha_g\otimes \gamma_g)(v_h^{(1)}) \\
& \quad -xu_{g}^{(1)}v_{g}^{(1)}(\alpha_g\otimes \gamma_g)(u_{h}^{(1)}v_{h}^{(1)})\|+\frac{\delta_1}{2}+
\frac{\delta_2}{2} \\
(4.6.5), (4.6.6)\; 
& < \| xu_g^{(1)}v_g^{(1)}(\alpha_g\otimes \gamma_g)(u_h^{(1)})v_g^{(1)*}
v_g^{(1)}(\alpha_g\otimes \gamma_g)(v_h^{(1)}) \\
& \quad -xu_{g}^{(1)}v_{g}^{(1)}(\alpha_g\otimes \gamma_g)(u_{h}^{(1)}v_{h}^{(1)})\|+\frac{\delta_1}{2}
+ \frac{\delta_2}{2}+ \frac{\delta_2\| x\|}{2}  \\
(4.6.4)\; &< \delta_1
\end{align*}
and 
\begin{align*}
 \| (\alpha_g\otimes \gamma_g)(x)- &\mathrm{Ad}(z_g^{(1)})\circ (\alpha_g\otimes \gamma_g)(x)\| \\
& = \|  (\alpha_g\otimes \gamma_g)(x)- 
u_g^{(1)}v_g^{(1)}(\alpha_g\otimes \gamma_g)(x)v_g^{(1)*}u_{g}^{(1)*}\| \\
(4.6.5),(4.6.6)\; 
& < \| (\alpha_g\otimes \gamma_g)(x)- u_g^{(1)}(\alpha_g\otimes\mathrm{id}_{\mathcal{W}})(x) 
u_{g}^{(1)*}\|+ \frac{\delta_2}{2} \\
(4.6.1)\; 
& <\frac{\delta_1}{2}+ \frac{\delta_2}{2}< \delta_1
\end{align*}
for any $x\in F_1^{\prime}$ and $g\in \Gamma$. 
Hence $(\mathrm{id}_{\mathcal{W}\otimes \mathcal{W}}, z^{(1)})$ is a proper 
$(\Gamma, F_1^{\prime}, \delta_1)$-approximate cocycle morphism from 
$(\mathcal{W}\otimes\mathcal{W}, \alpha\otimes\gamma)$ to 
$(\mathcal{W}\otimes\mathcal{W}, \alpha\otimes\gamma)$. 
Since we have  
$$
\| \pi_{\tau_{\mathcal{W}\otimes \mathcal{W}}}(z_k^{(1)})-1_{\pi_{\tau_{\mathcal{W}\otimes\mathcal{W}}}
(\mathcal{W}\otimes \mathcal{W})^{''}} \|_2 < \delta_1
$$
for any $k\in N(\tilde{\gamma})$ by (4.6.3), (4.6.4) and (4.6.8), 
Corollary \ref{cor:uniqueness} implies that 
there exists a unitary element $w_1$ in 
$(\mathcal{W}\otimes\mathcal{W})^{\sim}$ such that 
$$
\| x -w_1xw_1^* \| < \frac{\varepsilon_1}{2} \quad \text{and} \quad 
\|x(z_g^{(1)}-w_1(\alpha_g\otimes\gamma_g) (w_1^*)) \| < \frac{\varepsilon_1}{2}
$$
for any $x\in F_1$ and $g\in \Gamma$. Put 
$$
\psi_1:= \mathrm{Ad}(w_1^*)\quad \text{and} 
\quad
\tilde{v}^{(1)}_g:= w_1^*v_g^{(1)}(\alpha_g\otimes\gamma_g)(w_1)
$$
for any $g\in \Gamma$. Then $(\psi_1, \tilde{v}^{(1)})$ is a proper 
$(\Gamma, G_1, \varepsilon_1)$-approximate cocycle morphism from 
$(\mathcal{W}\otimes\mathcal{W}, \alpha\otimes\mathrm{id}_{\mathcal{W}})$ to 
$(\mathcal{W}\otimes\mathcal{W}, \alpha\otimes\gamma)$ such that 
\begin{align*}
\| \psi_1 \circ \varphi_1 (x) - x \|= \| w_1^*xw_1 -x\| < \frac{\varepsilon_1}{2}< \varepsilon_1
\end{align*}
and 
\begin{align*}
\| \psi_1\circ \varphi_1(x)(\psi_1(\tilde{u}_g^{(1)})\tilde{v}_g^{(1)}-1_{(\mathcal{W}\otimes\mathcal{W})^{\sim}})\| 
& =\| w_1^*x u_{g}^{(1)}v_{g}^{(1)}(\alpha_g\otimes\gamma_g)(w_1)- w_1^*xw_1\| \\
& =\| x z_{g}^{(1)}-xw_1(\alpha_g\otimes\gamma_g)(w_1^*)\| 
< \frac{\varepsilon_1}{2}<\varepsilon_1 
\end{align*}
for any $x\in F_1$ and $g\in \Gamma$. 

Step 3. Choose a finite self-adjoint subset $F_2$ of $\mathcal{W}\otimes \mathcal{W}$ such that 
$$
\{x_2\}\cup G_1\cup \psi_1 (G_1) \cup \bigcup_{g\in \Gamma}
(\alpha_g\otimes\gamma_g) (F_1)\tilde{v}_g^{(1)}\subset F_2
$$
and 
$$
\{\tilde{v}_g^{(1)}\; |\; g\in \Gamma\}\subset \{x+\lambda 1_{(\mathcal{W}\otimes\mathcal{W})^{\sim}}\; |\; 
x\in F_2, \lambda\in\mathbb{C}\}.  
$$
Applying Corollary \ref{cor:uniqueness} to $\alpha\otimes \gamma$,  
$F_2$, $\Gamma$ and $\varepsilon_2/2$, 
we obtain $F^{\prime}_2$ and $\delta_3>0$. We may assume 
$$
\delta_3 < \min\left\{ \dfrac{\delta_2}{1+\displaystyle{\max_{x\in 
G_1\cup G_1^{\prime}}\| x\|} }, \;\varepsilon_2
\right\}. \eqno{(4.6.9)}
$$
Choose a finite subset $F_2^{\prime\prime}$ of $\mathcal{W}\otimes \mathcal{W}$ such that 
$$
F_2\cup F_2^{\prime}\cup \bigcup_{g\in\Gamma} G_1^{\prime}v_g^{(1)}\subset F_2^{\prime\prime}
$$
and
$$
\{w_1\}\cup\{v_g^{(1)}\; |\; g\in \Gamma\}\subset \{x+\lambda 1_{(\mathcal{W}\otimes\mathcal{W})^{\sim}}\; |\; 
x\in F_2^{\prime\prime}, \lambda\in\mathbb{C}\}.  \eqno{(4.6.10)}
$$
By Corollary \ref{cor:existence} (i), there exists  a map $u^{(2)}$ from 
$\Gamma$ to $U((\mathcal{W}\otimes\mathcal{W})^{\sim})$ such that 
$$
\| (\alpha_g\otimes \gamma_g )(x) - \mathrm{Ad}(u^{(2)}_g)\circ 
(\alpha_g\otimes \mathrm{id}_{\mathcal{W}})(x) 
\|< \frac{\delta_3}{2},  \eqno{(4.6.11)}
$$
$$
\| x(u_{gh}^{(2)}- u_g^{(2)}(\alpha_g\otimes \mathrm{id}_{\mathcal{W}})(u_h^{(2)}))\| < 
\frac{\delta_3}{2}
$$
and 
$$
\| \pi_{\tau_{\mathcal{W}\otimes\mathcal{W}}}(u_k^{(2)}) - U_k\|_2 < \frac{\delta_3}{2} 
$$
for any $x\in F_2^{\prime\prime}$, $g,h\in \Gamma$ and $k\in N(\tilde{\gamma})$. 
For any $g\in \Gamma$, 
put $z_g^{(2)}:=v_g^{(1)} u_g^{(2)}\in U((\mathcal{W}\otimes\mathcal{W})^{\sim})$. 
By considering $z_g^{(2)}$ in place of $z_g^{(1)}$, $v_g^{(1)}$ in place of $u_g^{(1)}$ and 
$u_{g}^{(2)}$ in place of $v_g^{(1)}$ in arguments in Step 2, 
we obtain that  
$(\mathrm{id}_{\mathcal{W}\otimes \mathcal{W}}, z^{(2)})$ is a proper 
$(\Gamma, G_1^{\prime}, \delta_2)$-approximate cocycle morphism from 
$(\mathcal{W}\otimes\mathcal{W}, \alpha\otimes\mathrm{id}_{\mathcal{W}})$ to 
$(\mathcal{W}\otimes\mathcal{W}, \alpha\otimes\mathrm{id}_{\mathcal{W}})$ such that   
$$
\| \pi_{\tau_{\mathcal{W}\otimes \mathcal{W}}}(z_k^{(2)})-1_{\pi_{\tau_{\mathcal{W}\otimes\mathcal{W}}}
(\mathcal{W}\otimes\mathcal{W})^{''}} \|_2 < \delta_2
$$
for any $k\in N(\tilde{\gamma})$. 
Hence 
there exists a unitary element $w_2$ in 
$(\mathcal{W}\otimes\mathcal{W})^{\sim}$ such that 
$$
\| x -w_2xw_2^* \| < \frac{\varepsilon_1}{2} \quad \text{and} \quad 
\|x(z_g^{(2)}-w_2(\alpha_g\otimes\mathrm{id}_{\mathcal{W}}) (w_2^*)) \| < \frac{\varepsilon_1}{2}
$$
for any $x\in G_1$ and $g\in \Gamma$ by Corollary \ref{cor:uniqueness}. Put 
$$
\varphi_2:= \mathrm{Ad}(w_2^*w_1)\quad \text{and} 
\quad
\tilde{u}^{(2)}_g:= w_2^*w_1u_g^{(2)}
(\alpha_g\otimes\mathrm{id}_{\mathcal{W}})(w_1^*w_2)
$$
for any $g\in \Gamma$. 
Then $(\varphi_2, \tilde{u}^{(2)})$ is a proper 
$(\Gamma, F_2, \varepsilon_2)$-approximate cocycle morphism from 
$(\mathcal{W}\otimes\mathcal{W}, \alpha\otimes\gamma)$ to 
$(\mathcal{W}\otimes\mathcal{W}, \alpha\otimes\mathrm{id}_{\mathcal{W}})$ 
such that 
\begin{align*}
\| \varphi_2 \circ \psi_1 (x) - x \|= \| w_2^*xw_2 -x\| < \varepsilon_1
\end{align*}
for any $x\in G_1$. Furthermore, we have 
\begin{align*}
\|  \varphi_2\circ \psi_1(x) (\varphi_2(\tilde{v}_g^{(1)}&)\tilde{u}_g^{(2)}-1_{(\mathcal{W}\otimes\mathcal{W})^{\sim}})\| \\ 
& = \| w_2^*xv_g^{(1)}(\alpha_g\otimes \gamma_g)(w_1)u_g^{(2)}
(\alpha_g\otimes\mathrm{id}_{\mathcal{W}})(w_1^*w_2)- w_2^*xw_2 \| \\
(4.6.10), (4.6.11)\; & < \| w_2^*xv_g^{(1)}u_g^{(2)}
(\alpha_g\otimes\mathrm{id}_{\mathcal{W}})(w_2)-w_2^*xw_2\| 
+\frac{\|x\|  \delta_3}{2} \\
(4.6.9)\; & < \| xz_g^{(2)}-xw_2(\alpha_g\otimes\mathrm{id}_{\mathcal{W}})(w_2^*)\|+\frac{\delta_2}{2} \\
& < \frac{\varepsilon_1}{2}+\frac{\delta_2}{2}<\frac{\varepsilon_1}{2}+\frac{\delta_1}{2}<\varepsilon_1
\end{align*}
for any $x\in G_1$ and $g\in \Gamma$. 

Step 4. Choose a finite self-adjoint subset $G_2$ of $\mathcal{W}\otimes \mathcal{W}$ such that 
$$
F_2\cup \varphi_2 (F_2) \cup \bigcup_{g\in \Gamma}
(\alpha_g\otimes\mathrm{id}_{\mathcal{W}}) (G_1)\tilde{u}_g^{(2)}\subset G_2
$$
and 
$$
\{\tilde{u}_g^{(2)}\; |\; g\in \Gamma\}\subset \{x+\lambda 1_{(\mathcal{W}\otimes\mathcal{W})^{\sim}}\; |\; x\in G_2, \lambda\in\mathbb{C}\}.  
$$
By a similar way as in Step 2, we obtain $G_2^{\prime}$ and $\delta_4>0$ satisfying 
$$
\delta_4 < \dfrac{\delta_3}{1+\displaystyle{\max_{x\in F_2\cup F_2^{\prime}}\| x\|} }.
$$
Choose a finite subset $G_2^{\prime\prime}$ of $\mathcal{W}\otimes \mathcal{W}$ such that 
$$
G_2\cup G_2^{\prime}\cup \bigcup_{g\in\Gamma} F_2^{\prime}u_g^{(2)}\subset G_2^{\prime\prime}
$$
and
$$
\{w_1^*w_2 \}\cup\{u_g^{(2)}\; |\; g\in \Gamma\}\subset \{x+\lambda 1_{(\mathcal{W}\otimes\mathcal{W})^{\sim}}\; |\; 
x\in G_2^{\prime\prime}, \lambda\in\mathbb{C}\}.  
$$
Similar arguments as in Step 2 show that there exist a map $v^{(2)}$ from 
$\Gamma$ to $U((\mathcal{W}\otimes\mathcal{W})^{\sim})$ and a unitary
element $w_3$ in $(\mathcal{W}\otimes\mathcal{W})^{\sim}$ satisfying suitable properties. 
Put 
$$
\psi_2:= \mathrm{Ad}(w_3^*w_1^*w_2)\quad \text{and} 
\quad
\tilde{v}^{(2)}_g:= w_3^*w_1^*w_2v_g^{(2)}
(\alpha_g\otimes\mathrm{id}_{\mathcal{W}})(w_3^*w_1^*w_2)
$$
for any $g\in \Gamma$. 
Similar arguments as in Step 3 imply that $(\psi_2, \tilde{v}^{(2)})$ is a proper 
$(\Gamma, G_2, \varepsilon_2)$-approximate cocycle morphism from 
$(\mathcal{W}\otimes\mathcal{W}, \alpha\otimes\mathrm{id}_{\mathcal{W}})$ to 
$(\mathcal{W}\otimes\mathcal{W}, \alpha\otimes\gamma)$ 
such that 
$$
\| \psi_2\circ \varphi_2(x) - x \| < \frac{\varepsilon_2}{2}< \varepsilon_2
$$
and 
$$
\|  \psi_2\circ \varphi_2(x) (\psi_2(\tilde{u}_g^{(2)})\tilde{v}_g^{(2)}-1_{(\mathcal{W}\otimes\mathcal{W})^{\sim}})\|< \frac{\varepsilon_2}{2}
+\frac{\delta_3}{2}<\varepsilon_2
$$
for any $x\in F_2$ and $g\in \Gamma$. 

Repeating this process, we obtain increasing sequences of finite self-adjoint subsets 
$\{F_n\}_{n=1}^{\infty}$ and $\{G_n\}_{n=1}^{\infty}$ of $\mathcal{W}\otimes\mathcal{W}$ and 
sequences of proper approximate cocycle morphisms $\{(\varphi_n, \tilde{u}^{(n)})\}_{n=1}^{\infty}$ 
and $\{(\psi_n, \tilde{v}^{(n)})\}_{n=1}^{\infty}$ satisfying the assumption of 
\cite[Proposition 5.2]{Na5}. Consequently, $\alpha\otimes \gamma$ is cocycle conjugate to 
$\alpha\otimes \mathrm{id}_{\mathcal{W}}$ by \cite[Proposition 5.2]{Na5}.
\end{proof}

If $\alpha$ is a $\mathcal{W}$-absorbing action, then $\alpha$ is cocycle conjugate to 
$\alpha\otimes\mathrm{id}_{\mathcal{W}}$. Hence we obtain the following corollary by the theorem above 
and Theorem \ref{thm:property-W}. 

\begin{cor}\label{cor;absorption}
Let $\alpha$ and $\gamma$ be outer actions of a finite group $\Gamma$ on $\mathcal{W}$. 
Assume that $\alpha$ is $\mathcal{W}$-absorbing. 
If the characteristic invariant of $\tilde{\gamma}$ is trivial and $N(\tilde{\alpha})\subseteq N(\tilde{\gamma})$, 
then $\alpha\otimes \gamma$ is cocycle conjugate to $\alpha$.
\end{cor}

\section{Classification}\label{sec:cla}

In this section we shall completely classify outer 
$\mathcal{W}$-absorbing actions of finite groups on $\mathcal{W}$ up to conjugacy and 
cocycle conjugacy. 

The following theorem is one of the main theorems in this paper. 

\begin{thm}\label{thm:main-1}
Let $\alpha$ and $\beta$ be outer $\mathcal{W}$-absorbing actions of a finite group 
$\Gamma$ on 
$\mathcal{W}$. Then $\alpha$ is cocycle conjugate to $\beta$ if and only if 
$(N(\tilde{\alpha}), \Lambda (\tilde{\alpha}))=(N(\tilde{\beta}), \Lambda (\tilde{\beta}))$. 
\end{thm}
\begin{proof}
The only if part is obvious. We shall show the if part. 
Since $\alpha$ and $\beta$ are $\mathcal{W}$-absorbing actions, $\alpha$ and $\beta$ are 
cocycle conjugate to $\alpha\otimes\mathrm{id}_{\mathcal{W}}$ and 
$\beta\otimes\mathrm{id}_{\mathcal{W}}$, respectively. 
Define actions $\gamma$ and $\gamma^{\prime}$ of $\Gamma$ on $\mathcal{W}$ by 
$$
\gamma:= 
S^{(\Gamma, N(\tilde{\alpha}), \Lambda(\tilde{\alpha})^{-1})}\otimes \beta
\otimes\mathrm{id}_{\mathcal{W}} \quad \text{on} \quad
A_{(\Gamma, N(\tilde{\alpha}), \Lambda(\tilde{\alpha})^{-1})}
\otimes \mathcal{W} \otimes  \mathcal{W}\cong \mathcal{W}
$$ 
and 
$$
\gamma^{\prime}:= 
\alpha \otimes\mathrm{id}_{\mathcal{W}}\otimes S^{(\Gamma, N(\tilde{\alpha}),
\Lambda(\tilde{\alpha})^{-1})}\quad \text{on} \quad  
\mathcal{W}\otimes \mathcal{W}\otimes 
A_{(\Gamma, N(\tilde{\alpha}), \Lambda(\tilde{\alpha})^{-1})}\cong \mathcal{W}
$$ 
where $S^{(\Gamma, N(\tilde{\alpha}), \Lambda(\tilde{\alpha})^{-1})}$ is the outer action of 
$\Gamma$ on the simple unital monotracial AF algebra 
$A_{(\Gamma, N(\tilde{\alpha}), \Lambda(\tilde{\alpha})^{-1})}$
constructed in 
Proposition \ref{model:cocycle-conjugacy}. 
Then $\alpha\otimes\mathrm{id}_{\mathcal{W}}\otimes \gamma$ is equal to 
$\gamma^{\prime}\otimes \beta\otimes \mathrm{id}_{\mathcal{W}}$. 
It is easy to see that 
$\gamma$ and $\gamma^{\prime}$ are outer actions of $\Gamma$ on $\mathcal{W}$ such that 
the characteristic invariants of $\tilde{\gamma}$ and $\tilde{\gamma}^{\prime}$ 
are trivial and $N(\tilde{\gamma})=N(\tilde{\gamma}^{\prime})=N(\tilde{\alpha})
=N(\tilde{\beta})$. 
Furthermore, $\alpha\otimes \gamma$, 
$\gamma^{\prime}\otimes \beta$, $\gamma|_{C(g)}$
and $\gamma^{\prime}|_{C(g)}$ have property W for any $g\in \Gamma$ by Theorem 
\ref{thm:property-W}. 
Therefore we have 
\begin{align*}
\alpha
& \sim_{c.c} \alpha\otimes\mathrm{id}_{\mathcal{W}} \sim_{c.c} \alpha\otimes \gamma 
\sim_{c.c}
\alpha\otimes\mathrm{id}_{\mathcal{W}}\otimes \gamma \\
&= \gamma^{\prime}\otimes \beta \otimes 
\mathrm{id}_{\mathcal{W}} \sim_{c.c} \gamma^{\prime}\otimes\beta 
\sim_{c.c} \mathrm{id}_{\mathcal{W}}\otimes \beta \sim_{c.c} \beta 
\end{align*}
by Theorem \ref{thm;absorption} where the notation $\alpha \sim_{c.c} \beta$ means that $\alpha$ is 
cocycle conjugate to $\beta$. 
\end{proof}

The following corollary is an immediate consequence of the theorem above. 

\begin{cor}\label{cor:main-1}
Let $\alpha$ and $\beta$ be outer $\mathcal{W}$-absorbing actions of a finite group 
$\Gamma$ on 
$\mathcal{W}$. Then $\alpha$ is cocycle conjugate to $\beta$ if and only if 
$\tilde{\alpha}$ is cocycle conjugate to $\tilde{\beta}$. 
\end{cor}

We shall consider the classification up to conjugacy. 
The following lemma is an application of the construction of model actions in the classification up to 
cocycle conjugacy. 

\begin{lem}\label{lem:conjugacy}
Let $\alpha$ be an outer $\mathcal{W}$-absorbing action of a finite group $\Gamma$ on 
$\mathcal{W}$, and let $U$ be a map from $N(\tilde{\alpha})$ to 
$U(\pi_{\tau_{\mathcal{W}}}(\mathcal{W})^{''})$ such that 
$\tilde{\alpha}_h=\mathrm{Ad}(U_h)$ for any $h\in N(\tilde{\alpha})$. 
If $\eta$ is an element in $H^1(N(\tilde{\alpha}))^{\Gamma}$, 
then there exists an automorphism $\theta$ of 
$\mathcal{W}\otimes\mathbb{K}(\ell^2(\Gamma))$ 
such that 
$$
\theta \circ (\alpha_g\otimes \mathrm{Ad}(\rho_g)) =
(\alpha_g\otimes \mathrm{Ad}(\rho_g))\circ \theta \quad  
\text{and} 
\quad
\tilde{\theta}(U_h\otimes \rho_h)=\eta (h)(U_h\otimes \rho_h)
$$
for any $g\in \Gamma$ and $h\in N(\tilde{\alpha})$ where 
$\rho$ is the right regular representation of $\Gamma$ on 
$\ell^2(\Gamma)$. 
\end{lem}
\begin{proof}
Theorem \ref{thm:main-1} implies that $\alpha$ is cocycle conjugate to 
$S^{(\Gamma, N(\tilde{\alpha}), \Lambda(\tilde{\alpha}))}\otimes \mathrm{id}_{\mathcal{W}}$ 
where $S^{(\Gamma, N(\tilde{\alpha}), \Lambda(\tilde{\alpha}))}$ is the outer action of 
$\Gamma$ on the simple unital monotracial AF algebra 
$A_{(\Gamma, N(\tilde{\alpha}), \Lambda(\tilde{\alpha}))}$ constructed in 
Proposition \ref{model:cocycle-conjugacy}. 
Put 
$$
C:=A_{(\Gamma, N(\tilde{\alpha}),\Lambda(\tilde{\alpha}))}\otimes \mathcal{W}\otimes
\mathbb{K}(\ell^2(\Gamma))
\quad \text{and} \quad  
\alpha^{\prime}:= S^{(\Gamma, N(\tilde{\alpha}), \Lambda(\tilde{\alpha}))}\otimes
\mathrm{id}_{\mathcal{W}}\otimes\mathrm{Ad}(\rho).
$$ 
Since $\alpha\otimes \mathrm{Ad}(\rho)$ is conjugate to 
$\alpha^{\prime}$, 
there exists an isomorphism $\varphi$ from $\mathcal{W}\otimes\mathbb{K}(\ell^2(\Gamma))$  
onto $C$ such that  
$
\varphi \circ (\alpha_g\otimes \mathrm{Ad}(\rho_g))= 
\alpha^{\prime}_g\circ \varphi
$
for any $g\in \Gamma$. 
Define a map $W$ from $N(\tilde{\alpha})$ to 
$U(\pi_{\tau_{C}}(C)^{''})$ by 
$$
W_h:= \tilde{\varphi}(U_h\otimes \rho_h)
$$
for any $h\in N(\tilde{\alpha})$ where $\tilde{\varphi}$ is the induced isomorphism from 
$\pi_{\tau_{\mathcal{W}\otimes\mathbb{K}(\ell^2(\Gamma))}}
(\mathcal{W}\otimes\mathbb{K}(\ell^2(\Gamma)))^{''}$ onto $\pi_{\tau_{C}}(C)^{''}$ by $\varphi$. 
Then we have $\tilde{\alpha}^{\prime}_h=\mathrm{Ad}(W_h)$ for any $h\in N(\tilde{\alpha})$. 
Hence Corollary \ref{cor:1.5.9} implies that for any $\eta\in H^1(N(\tilde{\alpha}))^{\Gamma}$, 
there exists an automorphism $\beta$ of $C$ such that 
$$
\beta \circ \alpha^{\prime}_g=\alpha^{\prime}_g\circ \beta \quad  
\text{and} 
\quad
\tilde{\beta}(W_h)=\eta (h)W_h
$$
for any $g\in \Gamma$ and $h\in N(\tilde{\alpha})$. 
Put $\theta:= \varphi^{-1}\circ \beta \circ \varphi$. Then $\theta$ has the desired property. 
\end{proof}

The following theorem is one of the main theorems in this paper. 

\begin{thm}\label{thm:main-2}
Let $\alpha$ and $\beta$ be outer $\mathcal{W}$-absorbing actions of a finite group 
$\Gamma$ on 
$\mathcal{W}$. Then $\alpha$ is conjugate to $\beta$ if and only if 
$(N(\tilde{\alpha}), \Lambda (\tilde{\alpha}), i(\tilde{\alpha}))
=(N(\tilde{\beta}), \Lambda (\tilde{\beta}), i(\tilde{\beta}))$. 
\end{thm}
\begin{proof}
The only if part is obvious. We shall show the if part. It is enough to show that 
$\alpha\otimes \mathrm{id}_{\mathbb{K}(\ell^2(\Gamma))}$ is conjugate to 
$\beta\otimes \mathrm{id}_{\mathbb{K}(\ell^2(\Gamma))}$ 
because outer $\mathcal{W}$-absorbing actions $\alpha$ and $\beta$ are conjugate to 
$\alpha\otimes \mathrm{id}_{\mathbb{K}(\ell^2(\Gamma))}$ and
$\beta\otimes \mathrm{id}_{\mathbb{K}(\ell^2(\Gamma))}$, respectively  
by Proposition \ref{pro:def-w-absorbing} and the fact that 
$\mathcal{W}$ is isomorphic to $\mathcal{W}\otimes \mathbb{K}(\ell^2(\Gamma))$. 

Choose a pair 
$(\lambda, \mu)\in Z(\Gamma, N(\tilde{\alpha}))= Z(\Gamma, N(\tilde{\beta}))$ such that 
$[(\lambda, \mu)]=\Lambda(\tilde{\alpha})=\Lambda(\tilde{\beta})$.
Since we have $i(\tilde{\alpha})=i(\tilde{\beta})$, 
\cite[Proposition 1.4.6]{Jones} implies that there exist a map $V$ from $N(\tilde{\alpha})$ to 
$U(\pi_{\tau_{\mathcal{W}}}(\mathcal{W})^{''})$ and a unitary element $U$ in 
$\pi_{\tau_{\mathcal{W}}}(\mathcal{W})^{''}$ such that 
$$
\mathrm{Ad}(V_h)=\tilde{\alpha}_h, \quad \mathrm{Ad}(UV_hU^*)=\tilde{\beta}_h, \quad 
\tilde{\alpha}_g(V_{g^{-1}hg})=\lambda (g, h)V_{h}
$$
and 
$$
\tilde{\beta}_g(UV_{g^{-1}hg}U^*)= \lambda (g, h)UV_{h}U^*
$$
for any $g\in \Gamma$ and $h\in N(\tilde{\alpha})$. 
Put 
$$
W_h:= UV_hU^*
$$ 
for any $h\in N(\tilde{\beta})$. 

Theorem \ref{thm:main-1} implies that 
$\alpha\otimes \mathrm{Ad}(\rho)$ is conjugate to $\beta\otimes \mathrm{Ad}(\rho)$ 
where $\rho$ is the right regular representation of $\Gamma$ on 
$\ell^2(\Gamma)$, that is, 
there exists an automorphism $\theta_1$ of $\mathcal{W}\otimes \mathbb{K}(\ell^2(\Gamma))$ 
such that $\theta_1\circ (\alpha_g\otimes 
\mathrm{Ad}(\rho_g))= (\beta_g\otimes \mathrm{Ad}(\rho_g))\circ \theta_1$ for any 
$g\in \Gamma$. 
It follows from factoriality of $\pi_{\tau_{\mathcal{W}\otimes\mathbb{K}(\ell^2(\Gamma))}}
(\mathcal{W}\otimes\mathbb{K}(\ell^2(\Gamma)))^{''}$ that there exists an element 
$\eta$ in $H^1(N(\tilde{\alpha}))^{\Gamma}$ such that 
$$
\tilde{\theta}_1(V_h\otimes \rho_h)= \eta (h) (W_h\otimes \rho_h)
$$
for any $h\in N(\tilde{\alpha})$. By Lemma \ref{lem:conjugacy}, there exists an automorphism $\theta_2$ 
of  $\mathcal{W}\otimes \mathbb{K}(\ell^2(\Gamma))$
$$
\theta_2 \circ (\beta_g\otimes \mathrm{Ad}(\rho_g)) =
(\beta_g\otimes \mathrm{Ad}(\rho_g))\circ \theta_2 \quad  
\text{and} 
\quad
\tilde{\theta}_2(W_h\otimes \rho_h)=\overline{\eta (h)}(W_h\otimes \rho_h)
$$
for any $g\in \Gamma$ and $h\in N(\tilde{\beta})$. (Note that if $\eta\in H^1(N(\tilde{\alpha}))^{\Gamma}$, 
then $\overline{\eta}\in H^1(N(\tilde{\alpha}))^{\Gamma}$.)
 Put $\theta:= \theta_2\circ \theta_1$. 
Then $\theta$ is an automorphism of $\mathcal{W}\otimes \mathbb{K}(\ell^2(\Gamma))$ such that 
$$
\theta\circ (\alpha_g\otimes \mathrm{Ad}(\rho_g))= (\beta_g\otimes 
\mathrm{Ad}(\rho_g))\circ \theta
\quad 
\text{and}
\quad
\tilde{\theta}(V_h\otimes \rho_h)= W_h\otimes \rho_h
$$
for any $g\in \Gamma$ and $h\in N(\tilde{\alpha})$. 

Define a map $u$ from $\Gamma$ to $U(M(\mathcal{W}\otimes \mathbb{K}(\ell^2(\Gamma))))$ by 
$$
u_g:= \theta (1_{M(\mathcal{W})}\otimes \rho_g^*)(1_{M(\mathcal{W})}\otimes \rho_g)
$$
for any $g\in \Gamma$. By the same computations as in the proof of 
\cite[Theorem 1.4.8]{Jones} (see \cite[Section 6.3]{Jones}), we see that $u$ is a 
$\beta\otimes \mathrm{id}_{\mathbb{K}(\ell^2(\Gamma))}$-cocycle such that 
$$
\theta \circ (\alpha_g\otimes \mathrm{id}_{\mathbb{K}(\ell^2(\Gamma))})
= \mathrm{Ad}(u_g)\circ (\beta_g\otimes \mathrm{id}_{\mathbb{K}(\ell^2(\Gamma))})\circ \theta
$$
for any $g\in \Gamma$. In particular, $\alpha\otimes \mathrm{id}_{\mathbb{K}(\ell^2(\Gamma))}$ 
is conjugate to $\mathrm{Ad}(u)\circ (\beta\otimes \mathrm{id}_{\mathbb{K}(\ell^2(\Gamma))})$. 
Since we have 
\begin{align*}
&\tilde{\tau}_{\mathcal{W}\otimes \mathbb{K}(\ell^2(\Gamma))}(
\pi_{\tau_{\mathcal{W}\otimes \mathbb{K}(\ell^2(\Gamma))}}(u_h)(W_h\otimes 1_{\mathbb{K}(\ell^2(\Gamma))})) 
\\
&= \tilde{\tau}_{\mathcal{W}\otimes \mathbb{K}(\ell^2(\Gamma))}(\tilde{\theta}
(1_{M(\mathcal{W})}\otimes \rho_h^*)(W_h\otimes \rho_h)) \\
&= \tilde{\tau}_{\mathcal{W}\otimes \mathbb{K}(\ell^2(\Gamma))}(\tilde{\theta}
(V_h\otimes 1_{\mathbb{K}(\ell^2(\Gamma))})) \\
&=  \tilde{\tau}_{\mathcal{W}\otimes \mathbb{K}(\ell^2(\Gamma))}(V_h\otimes 
1_{\mathbb{K}(\ell^2(\Gamma))}) \\
&= \tilde{\tau}_{\mathcal{W}\otimes \mathbb{K}(\ell^2(\Gamma))}
((U\otimes 1_{\mathbb{K}(\ell^2(\Gamma))})
(V_h\otimes 1_{\mathbb{K}(\ell^2(\Gamma))})(U\otimes 1_{\mathbb{K}(\ell^2(\Gamma))})^*) \\
&=\tilde{\tau}_{\mathcal{W}\otimes \mathbb{K}(\ell^2(\Gamma))}(W_h\otimes 1_{\mathbb{K}(\ell^2(\Gamma))}) 
\end{align*}
for any $h\in N(\tilde{\beta})$, Proposition \ref{pro:coboundary} implies that 
$\mathrm{Ad}(u)\circ (\beta\otimes \mathrm{id}_{\mathbb{K}(\ell^2(\Gamma))})$ is 
conjugate to $\beta\otimes \mathrm{id}_{\mathbb{K}(\ell^2(\Gamma))}$. 
Consequently, $\alpha\otimes \mathrm{id}_{\mathbb{K}(\ell^2(\Gamma))}$ is conjugate to 
$\beta\otimes \mathrm{id}_{\mathbb{K}(\ell^2(\Gamma))}$. 
\end{proof}

The following corollaries are immediate consequences of the theorem above (and Theorem 
\ref{thm:model} for Corollary \ref{cor:model-conj}). 

\begin{cor}\label{cor:main-2}
Let $\alpha$ and $\beta$ be outer $\mathcal{W}$-absorbing actions of a finite group 
$\Gamma$ on 
$\mathcal{W}$. Then $\alpha$ is conjugate to $\beta$ if and only if 
$\tilde{\alpha}$ is conjugate to $\tilde{\beta}$. 
\end{cor}

\begin{cor}\label{cor:model-conj}
Let $\alpha$ be an outer $\mathcal{W}$-absorbing action of a finite group 
$\Gamma$ on $\mathcal{W}$. Then there exists a probability measure $m$ with full support on 
$P_{\lambda_{\tilde{\alpha}}, \mu_{\tilde{\alpha}}}$ such that $\alpha$ is conjugate to  
$\alpha^{(\Gamma, N(\tilde{\alpha}), \Lambda(\tilde{\alpha}), m)}\otimes \mathrm{id}_{\mathcal{W}}$ 
on $A_{\Gamma, N(\tilde{\alpha}), \Lambda(\tilde{\alpha}), m}\otimes \mathcal{W}$.  
\end{cor}

\section*{Acknowledgments}
The author would like to thank Yoshikazu Katayama for a helpful discussion about the proof of 
Theorem \ref{thm:main-1}.

\end{document}